\documentclass[american,parskip=half,11pt,
]{scrartcl}

\usepackage{amssymb,amsmath,mathtools}
\usepackage[hyperref,amsmath,thmmarks]{ntheorem} 
\usepackage{aliascnt}
\usepackage{dashrule}

\usepackage{caption}
\usepackage{subfigure}
\usepackage{multirow}
\usepackage{booktabs}

\usepackage{graphicx}
\usepackage[justification=raggedleft]{subcaption}
\usepackage{xcolor}
\usepackage{float}
\mathtoolsset{showonlyrefs}

\usepackage[backend=biber, style=numeric-comp, 
]{biblatex}

\usepackage[ruled,vlined]{algorithm2e}

\usepackage[english,pdfusetitle,colorlinks=true,linkcolor=blue,citecolor=green!50!black,urlcolor=blue,bookmarks=true,
]{hyperref}

\usepackage{mathabx}

\usepackage{hyperref}

\definecolor{ppurple}{RGB}{238,51,119}
\definecolor{bblack}{RGB}{66,66,66}
\definecolor{ggray}{RGB}{153,153,153}
\definecolor{rred}{RGB}{204,51,17}
\definecolor{oorange}{RGB}{255,112,67}
\definecolor{ggreen}{RGB}{0,153,136}
\definecolor{bblue}{RGB}{0,119,187}

\theoremseparator{.}
\newtheorem{lemma}{Lemma}

\newaliascnt{proposition}{lemma}

\aliascntresetthe{proposition}

\newaliascnt{corollary}{lemma}
\newtheorem{corollary}[corollary]{Corollary}
\aliascntresetthe{corollary}

\newaliascnt{theorem}{lemma}
\newtheorem{theorem}[theorem]{Theorem}
\aliascntresetthe{theorem}

\theorembodyfont{\normalfont}
\newaliascnt{definition}{lemma}
\newtheorem{definition}[definition]{Definition}
\aliascntresetthe{definition}

\newaliascnt{assumption}{lemma}

\aliascntresetthe{assumption}

\newaliascnt{claim}{lemma}

\aliascntresetthe{claim}

\newaliascnt{notation}{lemma}

\aliascntresetthe{notation}

\newaliascnt{experiment}{lemma}

\aliascntresetthe{experiment}

\theoremsymbol{\ensuremath{\square}}

\newaliascnt{example}{lemma}
\newtheorem{example}[example]{Example}
\aliascntresetthe{example}

\newaliascnt{remark}{lemma}
\newtheorem{remark}[remark]{Remark}
\aliascntresetthe{remark}

\theoremseparator{:}
\theoremstyle{nonumberplain}
\theoremheaderfont{\normalfont\itshape}
\newtheorem{proof}{Proof}

\usepackage[draft, todonotes={textsize=tiny}]{changes}
\setuptodonotes{color=red, backgroundcolor=red!60!white,
  bordercolor=red, tickmarkheight=10pt}
  
\newcommand{\N}{\ensuremath{\mathbb{N}}}

\newcommand{\R}{\ensuremath{\mathbb{R}}}

\newcommand{\zb}[1]{\ensuremath{\boldsymbol{#1}}}

\DeclareMathOperator*{\diag}{diag}

\DeclareMathOperator*{\supp}{supp}

\DeclareMathOperator{\dist}{dist}
\DeclareMathOperator{\spn}{span}
\DeclareMathOperator{\KL}{KL}

\newcommand{\norm}[1]{\left\lVert #1
	\right\rVert}

\newcommand\q[1]{\mathcal{#1}} 
\newcommand\bb[1]{\mathbb{#1}} 

\newcommand{\tT}{\mathrm{T}}

\newcommand{\uX}{\underline{X}}
\newcommand{\oX}{\overline{X}}

\newcommand{\pd}{\mathrm{Sym}_{\succ 0}}
\newcommand{\psd}{\mathrm{Sym}_{\succeq 0}}

\usepackage{dsfont} 
\DeclareMathOperator{\1}{\mathds{1}}

\DeclareMathOperator{\Fix}{Fix}

\title{Unsupervised Ground Metric Learning}

\date{}
\author{%
	Janis Auffenberg \footnotemark[2]
	\and
	Jonas Bresch \footnotemark[2] \\   
	{\footnotesize\href{mailto:bresch@math.tu-berlin.de}{bresch@math.tu-berlin.de}}
	\and
	Oleh Melnyk\thanks{Corresponding author.} \thanks{Technische Universität Berlin, Institute of Mathematics, Straße des 17. Juni 136, 10623 Berlin, Germany} \thanks{Ludwig-Maximilians-Universität München, Bavarian AI Chair for Mathematical Foundation of Artificial Intelligence, Akademiestr. 7, 80799 Munich, Germany.} \\%
	{\footnotesize\href{mailto:melnyk@math.tu-berlin.de}{melnyk@math.tu-berlin.de}}
	\and
	Gabriele Steidl\footnotemark[2]\\%
	{\footnotesize\href{mailto:steidl@math.tu-berlin.de}{steidl@math.tu-berlin.de}}
}

\begin{document}
	
    \def\sectionautorefname{Section}
    \def\subsectionautorefname{Section} 
    
    \maketitle
    
    \begin{abstract}
       Data classification without access to labeled samples remains a challenging problem. It usually depends
       on an appropriately chosen distance between features, a topic addressed in metric learning.
       Recently, Huizing, Cantini and Peyr\'e proposed to simultaneously learn optimal transport (OT) cost matrices between samples and features of the dataset. This leads to the task of finding positive eigenvectors of a certain nonlinear function that maps cost matrices to OT distances.
       Having this basic idea in mind, we consider both the algorithmic and the modeling part of unsupervised metric learning. First, we examine appropriate algorithms and their convergence. In particular, we propose to use the stochastic random function iteration algorithm and prove that it converges linearly for our setting, although our operators are not paracontractive as it was required for convergence so far. 
       Second, we ask the natural question if the OT distance can be replaced by other distances. We show how  Mahalanobis-like distances fit into our considerations. Further, we examine an approach via graph Laplacians. In contrast to the previous settings, we have just to deal with linear functions in the wanted matrices here, so that simple algorithms from linear algebra can be applied. 
    \end{abstract}
    
\section{Introduction}  \label{sec: intro}
The goal of classification and clustering is to assign labels to a dataset consisting of $n$ samples $\{X_i\}_{i = 1}^n \subset \R^m$ with $m$ features each. 
    Typically, algorithms are based on the distances between features and the choice of distances significantly impacts the performance of clustering algorithms. Therefore, finding an optimal distance for a given labeled dataset is crucial.
    Since the work \cite{xing2002distance}, this problem has grown into a field of its own, known as metric learning. The applications of metric learning include computer vision \cite{hu2016deep, coskun2018human, milbich2020diva}, computational biology \cite{makrodimitris2019metric, bellazzi2021gene, huizing2022unsupervised, luerig2024bioencoder}, natural language processing \cite{kusner2015from, huang2016supervised, coria2020metric, biswas2022geometric}.
    Commonly, the search space is a class of distance functions $\dist_A$ parameterized by a matrix $A \in \R^{m \times m}$ and the task is reduced to searching for an optimal $A$.     
    A prominent example is Mahalanobis distance learning, where finding $A$ can be performed by, e.g., the large-margin nearest neighbors method \cite{weinberger2005distance} or information-theoretic metric learning \cite{davis2007information}. A detailed review on Mahalanobis distance learning can be found in \cite{bellet2015metric, suarez2021tutorial, ghojogh2022spectral, ghojogh2023elements}. Another class of distances is based on optimal transport (OT) cost matrices. 
Finding the cost matrix $A$ in the OT distance $W_A$ was initially proposed in \cite{cuturi2014ground}. 
While verifying that $A$ is a metric matrix is computationally heavy, in \cite{wang2012supervised} one of the constraints is dropped to speed up the learning process. In \cite{xu2018multi,kerdoncuff2020metric}, the entries of $A$ were assumed to be Mahalanobis distances $M_B$ between feature vectors and the matrix $B$ is learned instead, and in \cite{stuart2020inverse}, the metric matrix is associated to the shortest path of the underlying data graph. A similar relation is used in \cite{heitz2021ground} with graph geodesic distances.

Another class of distances uses the embedding of samples into a latent space
    \begin{equation}\label{eq: embedded distance}
        \dist_{\varphi}(X_i,X_j) \coloneqq \| \varphi(X_i) - \varphi(X_j) \|_2.
    \end{equation}
    A prominent class of methods for finding such embeddings is deep metric learning \cite{hadsell2006dimensionality}, where $\varphi$ is a network $\q N_\theta$ parametrized by $\theta$. Note that Euclidean distance in \eqref{eq: embedded distance} can be replaced by Wasserstein distance \cite{dou2022optimal} or completely avoided by directly maximizing the performance of the modified clustering algorithms on the embeddings $\q N_\theta(X_i)$ \cite{liu2017sphereface, wang2018cosface}. We refer the reader to surveys on deep metric learning \cite{kaya2019deep, vu2021deep, ghojogh2022spectral} for more details. 

    The general strategy for finding $A$ or $\theta$ in all of the above methods is to minimize $\dist(X_i, X_j)$ whenever $X_i$ and $X_j$ belong to the same class and maximize $\dist(X_i, X_j)$ otherwise.  Alternatively, losses based on triplets $(X_i, X_j, X_s)$ are employed \cite{weinberger2005distance, wang2012supervised, schroff2015facenet}, where $X_i$ and $X_j$ belongs to the same class, while $X_i$ and $X_s$ do not.          
    However, if the class labels are not available, it is no longer possible to apply supervised methods. 
    A possible solution is to avoid metric learning and heuristically choose the entries of $A$ based on features, with some application-motivated embedding $\varphi$. 
     For instance, the Word Moving \cite{kusner2015from} and Gene Mover \cite{bellazzi2021gene} distances use word and gene embeddings to compute $A$ for $W_A$. Other options are hierarchical neighbor embedding \cite{liu2024hierarchical} or optimal transport-based metric space embedding \cite{beier2025joint,beckmann2025normalized,beckmann2025maxnormalized}.  
    To circumvent the absence of labels for image clustering, some contrastive learning methods \cite{khodadadeh2019unsupervised, chen2020simple} use data augmentation. Namely, it is observed that by sampling a few images from the dataset, the images will likely have distinct labels. For each sampled image, a collection of images sharing the same unknown label is created through augmentation and is used for training.
    In \cite{iscen2018mining}, a manifold similarity distance is used for constructing triples $(X_i, X_j, X_s)$.

This paper is based on an unsupervised approach for OT metric learning proposed by Huizing, Cantini and Peyr\'e in  \cite{huizing2022unsupervised}.
The aim of this paper is twofold: first, we consider numerical algorithms to solve the relevant fixed-point problems, including
conditions under which the algorithms converge. In particular, we give a proof of linear convergence of the stochastic random function algorithm \cite{hermer2019random} also for non-paracontractive operators.
Later, we will verify how these conditions are fulfilled for the OT as well as for the Mahalanobis-like distance setting.
Second, we examine how the idea from \cite{huizing2022unsupervised} of simultaneously learning OT distance/cost matrices between samples and features
can be generalized to other distance matrices than just OT costs.
For modeling, we deal with Mahalanobis-like distances and propose an additional approach using graph Laplacians. The latter corresponds to linear functions in the matrices we aim to learn, so that numerical methods from linear algebra can be applied, making the approach much faster than the other non-linear ones.

\paragraph{Outline of the paper.} In Section \ref{sec: fix point}, we introduce two fixed point problems for general functions $\mathcal F$ and $G$ mapping from $\R^{m \times m}$  to $\R^{n \times n}$ and back.
In Subsection \ref{subsec:algs}, we propose a simple
iteration algorithms for each of the problems and address conditions on $\q F$ and $\q G$ such that a unique fixed point exists and convergence of the iteration sequence is guaranteed.
In Subsection \ref{subsec:stoch_algs}, we proposed a stochastic
algorithm, the so-called Random Function Iteration algorithm, for our second problem. We cannot use the convergence results shown for this algorithm in \cite{hermer2019random}, since our operators are not paracontractive, which is the essential assumption. Therefore, we provide our own convergence proof.
In the next sections, we specify the functions $\q F$ and $\q G$.
In each case, it must be ensured that the iterations map into a certain
subset of matrices, which is different for the different models.
In Section \ref{sec: wasserstein}, we deal with regularized OT-like
distances, in particular, the Sinkhorn divergence \cite{feydy2019interpolating}.
This setting was handled for problem \eqref{fix2} in the original paper \cite{huizing2022unsupervised}. We consider the problem again, complete the gaps in theory and
address previously established statements 
from our point of view, see Remark \ref{peyre_1}  and  Remark \ref{rem:peyre_2}. 
Then, in Section \ref{sec: kernel}, we deal with simpler model of regularized 
Mahalanobis-like distances.
Finally, we use just graph Laplacians for $\q F$ and $\q G$ for which problem \eqref{fix1} becomes a linear one. We prove the existence of a single largest eigenvalue of an associated linear eigenvalue problem.
Numerical results in Section \ref{sec:numerics} explore the performance of the proposed models in unsupervised clustering/classification tasks and highlight their differences.

\section{Fixed Point Algorithms}     \label{sec: fix point}
Let $[m] \coloneqq \{1,\ldots,m\}$. By $I_m$ we denote the $m \times m$ identity matrix and by $\1_m \in \bb R^m$ the vector with all entries equal to one.
Further, we use for $A \in \R^{m \times m}$ the notation
    \[
    \| A \|_p = \Big( \sum_{i,j=1}^m  |A_{i,j}|^p \Big)^{1/p}, \ 1 \le p < \infty, \qquad \|A\|_\infty = \max_{i,j\in [m]} |A_{i,j}|.
    \]
For $p=2$, this norm is also known as the Frobenius norm. 

For two functions 
$\mathcal F: \R^{m \times m}  \to \R^{n \times n}$ 
and
$\mathcal G: \R^{n \times n}  \to \R^{m \times m}$, which will be specified in the later sections, 
we are interested in matrices $A\in \mathbb R^{m \times m}$ and $B\in \mathbb R^{n \times n}$ fulfilling the following two  fixed-point relations \eqref{fix2} and \eqref{fix1}.
We include the first problem, since it was handled in the literature, in particular in \cite{huizing2022unsupervised}, while our main interest will be on the second problem.
First, we are asking for solutions of 
\begin{equation} \label{fix2}
    B =  \tilde {\q F}(A) \coloneqq \frac{ \q F(A)}{\|\q F(A)\|_\infty } 
    \quad \text{and} \quad 
    A = \tilde {\q G}(B) \coloneqq\frac{\q G(B)}{\|\q G(B)\|_\infty }. 
\end{equation}
Then we have 
\begin{equation}    \label{eq: operator T}
    A = \tilde{\q T}(A) , \quad  \tilde{\q T} \coloneqq \tilde{\q G}\circ \tilde {\q F}
\end{equation}
and the pair
$(A,B) \coloneqq \big(A, \tilde{ \q F}(A)\big)$
is a fixed point of
$(A,B) \mapsto \big(\tilde{\q G}(B), \tilde{\q F}(A) \big)$
or, in other words, an eigenvector of this operator with eigenvalue $1$.

Further, we will deal with problems of the form
\begin{equation} \label{fix1}
    B = \gamma_\q F \q F(A) 
    \quad \text{and} \quad 
    A = \gamma_\q G  \q G(B), \quad 
    \gamma_\q F, \gamma_\q G > 0,
\end{equation}
where $\gamma_\q F$ and $\gamma_\q G$ may depend on the data $X$, but
not on $A$ or $B$.
Note that \eqref{fix1} implies $A = \gamma_\q G  \q G \left( \gamma_\q F \q F(A) \right)$
and for a fixed point $A$ we have that 
$(A,B) \coloneqq (A, \gamma_{\q F} \q F(A))$ is a fixed point of 
$(A,B) \mapsto \big(\gamma_{\q G} \q G(B), \gamma_{\q F} \q F(A) \big)$. For $\gamma_\q F = \gamma_\q G = \gamma$ it is an eigenvector of $(A,B) \mapsto \big( \q G(B), \q F(A) \big)$
with eigenvalue $1/\gamma$.
To solve \eqref{fix1}, we consider the operator
$\q T: \R^{m \times m}  \to \R^{m \times m}$ be defined by
    \begin{equation}\label{eq: operator Q}
        \q T(A)
        \coloneqq 
        (1-\alpha) A
        + \alpha \gamma_\q G \q G \left( \gamma_\q F \q F (A) \right), \quad \alpha \in (0,1]
    \end{equation}  
    which is just the concatenation 
    $(\gamma_\q G  \q G) \circ  (\gamma_\q F \q F)$ for $\alpha = 1$.
   It is easy to check that the fixed point sets fulfill
    $$\Fix(\q T) = \Fix\big((\gamma_\q G \q G) \circ (\gamma_\q F \q F) \big) \quad \text{for all} \quad \alpha \in (0,1].
    $$

In the following, we propose a simple deterministic and a stochastic algorithm for finding fixed points and prove convergence results.  In Sections \ref{sec: wasserstein} and \ref{sec: kernel}, we will apply these results for various
nonlinear functions $\q F$ and $\q G$.
In the next Subsection \ref{subsec:algs}, we consider simple deterministic algorithms for solving \eqref{fix1} and \eqref{fix2}.
Then, in Subsection  \ref{subsec:stoch_algs}, we propose a stochastic algorithm for \eqref{fix1} which is better suited for high-dimensional problems and its prove linear convergence.

\subsection{Simple Fixed Point Iterations}\label{subsec:algs}
For solving \eqref{fix2}, we consider the following Algorithm \ref{alg:fix2}.

\begin{algorithm}[!h] 
        \caption{Fix-Point Algorithm for \eqref{fix2}} \label{alg:fix2}
        \vspace{1mm}
        \KwData{Initialization $A^0 \in \R^{m \times m}$.}     
        \For{$t=0,1,\ldots$}{
         $B^{t+1} =  \q F(A^t)/ \| \q F(A^t)\|_\infty$\\
         $A^{t+1} =  \q G(B^{t+1})/ \| \q G(B^{t+1})\|_\infty$  
                
        }       
    \end{algorithm}

Recall that an operator $\q F : \R^{m \times m} \to \R^{n \times n}$ is called \emph{Lipschitz continuous with Lipschitz constant} $L$ with respect to the norms $\| \cdot \|$, if
    \[
        \| \q F(A) - \q F(A')  \| 
        \le L \| A - A' \|
        \quad 
        \text{for all}
        \quad 
        A, A' \in \R^{m \times m}.
    \]
We will simply say that $\q F$ is $L$-\emph{ Lipschitz}.
If $L\le 1$, then  $\q F$ is referred to as \emph{nonexpansive}, 
and if $L<1$ as  \emph{contractive}.
We have the following convergence result.

\begin{theorem} \label{l:lipshcitz}
 Let $\q F$ and $\q G$ be Lipschitz continuous with constants $L_{\q F}$ and $L_{\q G}$
 and let  $\| \q F(A) \|_\infty \ge C_{\q F} >0$, $\| \q G(B) \|_\infty \ge C_{\q G} > 0$ 
    for all $A \in \R^{m \times m}$ 
    and $B \in \R^{n \times n}$.
    Then, 
    $\tilde{\q F}$ and 
    $\tilde{\q G}$ 
    defined by \eqref{fix2}
    are Lipschitz continuous with constants 
    $L_{\tilde{\q F}} \coloneqq 2 L_{\q F} / C_{\q F}$ 
    and 
    $L_{\tilde{ \q G}}  \coloneqq 2 L_{\q G} / C_{\q G}$ 
    and
     $\tilde{\q T} \coloneqq \tilde{\q G} \circ \tilde{\q F}$ 
    is Lipschitz continuous with constant 
    $L_{\tilde {\q T}} \coloneqq L_{\tilde {\q F}} L_{\tilde {\q G}}$. 
    If $L_{\tilde {\q T}} < \frac14 C_{\q F} C_{\q G}$, then
    there exists a unique fixed point $A$ of 
    $ \tilde{\q T}$
    and the sequence $(A^t)_{t \in \N}$ generated 
    by Algorithm \ref{alg:fix2}
    converges for an arbitrary initialization $A^0 \in \R^{m \times m}$     linearly to $A$,
     meaning that
    \begin{equation} \label{1}
    \|A^{t+1} - A \|_\infty \le L_{\tilde{\q T}} \| A^t - A \|_\infty
    \quad \text{for all} \quad
    t \in \N.
    \end{equation}  
    \end{theorem}     

\begin{proof} We have only to estimate the Lipschitz constant of $\tilde{\q F}$. The estimation for $\tilde{\q G}$ follows similarly and the other assertions are just conclusions from Banach's Fixed Point Theorem \ref{thm:banach}.
    For arbitrary $A,A' \in \R^{m \times m}$, we get
    \begin{align*}
        \tilde{\q F}(A) - \tilde{\q F}(A') 
        & = \frac{\q F(A)}{\| \q F(A) \|_\infty} - \frac{\q F(A')}{\| \q F(A') \|_\infty} \\
        & = \frac{ \| \q F(A') \|_\infty \, \q F(A) - \|\q F(A)\|_\infty \, \q F(A')}{\| \q F(A) \|_\infty \|\q F(A')\|_\infty}\\
        &= \frac{ [\| \q F(A') \|_\infty - \| \q F(A) \|_\infty] \, \q F(A) + \|\q F(A)\|_\infty [\q F(A) -\q F(A') ] }{\| \q F(A) \|_\infty \|\q F(A')\|_\infty}.
    \end{align*}
    Taking the norm together with reverse- and triangle inequalities yields
    \begin{align*}
        \left\| \tilde{\q F}(A) -\tilde{\q F}(A') \right\|_\infty 
        &\le \frac{ | \| \q F(A') \|_\infty - \| \q F(A) \|_\infty | \, \| \q F(A) \|_\infty + \|\q F(A)\|_\infty \| \q F(A) -\q F(A') \|_\infty }{\| \q F(A) \|_\infty \|\q F(A')\|_\infty} \\
        &\le \frac{ 2 \| \q F(A') - \q F(A) \|_\infty}{ \|\q F(A')\|_\infty} 
        \le \frac{ 2 L_{\q F} }{C_{\q F}} \| A - A' \|_\infty
        = L_{\tilde{\q F}} \| A - A' \|_\infty. 
    \end{align*}
     \end{proof}
     
For problem \eqref{fix1}, we propose Algorithm \ref{alg:fix}, for which we have the following convergence result.

\begin{algorithm}[!h] 
        \caption{Fix-Point Algorithm for \eqref{fix1}} \label{alg:fix}
        \vspace{1mm}
        \KwData{Initialization $A^0 \in \R^{m \times m}$, 
        parameter: $0<\alpha \leq 1$.}
        \For{$t=0,1,\ldots$}{
         $B^{t+1} = \gamma_{\q F} \q F(A^t)$\\
         $A^{t+1} = (1-\alpha)A^t + \alpha \gamma_{\q G}  \q  G  
        (B^{t+1}) $ 
        
        }       
    \end{algorithm}

\begin{theorem}\label{thm: convergence grad}
    Let $\q F$ and $\q G$ be Lipschitz continuous with constants $L_{\q F}$ and $L_{\q G}$. 
    Then the operator  $\q T: \R^{m \times m}  \to \R^{m \times m}$  defined by
    \eqref{eq: operator Q}
    is Lipschitz continuous with constant 
    $L_\q T \coloneqq 1 - \alpha \left(1 - L_\q F L_\q G\gamma_\q F \gamma_\q G \right)$.
    If $L_\q F L_\q G < \frac{1}{\gamma_\q F \gamma_\q G}$, 
    then there exists a unique fixed point $A$ of $\q T$ 
    and the sequence $(A^t)_{t \in \N}$ generated by Algorithm \ref{alg:fix} converges
    for an arbitrary initialization $A^0 \in \R^{m \times m}$ 
    linearly to $A$.
        
\end{theorem} 
    
    \begin{proof}
        Straightforward computations yield
        \begin{align*}
        \| \q T(A) - \q T(A') \|_\infty 
        & \le 
        (1 -\alpha) \| A - A' \|_\infty + \alpha         \gamma_\q G	\| \q G( \gamma_\q F \q F (A) ) - \q G( \q F (A') ) \|_\infty \\
        & 
        \le (1 -\alpha) \| A - A' \|_\infty + \alpha \gamma_\q G L_\q G \gamma_\q F L_\q F \| A - A'\|_\infty
        \le 
        L_{\q T}  \| A - A'\|_\infty.
        \end{align*}
    Thus,  for $L_\q F L_\q G < \frac{1}{\gamma_F \gamma_\q G}$ the operator 
    $\q Q$ is contractive
    and the existence of a unique fixed point as well as \eqref{1} follow from Banach's Fixed Point Theorem \ref{thm:banach}.
           \end{proof}

\begin{remark}\label{rem:pos_def}
At first glance it seems
    that the lower bounds 
    $\|\q F\|_\infty \ge C_\q F >0$ 
    and $\|\q G\|_\infty \ge C_\q G > 0$
    are not required in Theorem \ref{thm: convergence grad}.
    Yet,     vanishing $\q F$ and $\q G$ may still cause a problem 
    in finding non-trivial solutions. Consider for example positively homogeneous functions
    $\q F$ and $\q G$,  i.e.
        \[
        \q F(\lambda A) = \lambda \q F(A)
        \quad \text{and} \quad
        \q G(\lambda B) = \lambda \q G(B) 
        \quad \text{for all} \quad  
        \lambda>0.
        \]
    which fulfill the assumptions of Theorem \ref{thm: convergence grad}.
        Then, 
        we get for the $m \times m$ zero matrix $0_{m,m}$ that
        \[
        \q F(0_{m,m}) = \q F(\lambda 0_{m,m})  = \lambda \q F(0_{m,m}),
        \quad \text{for all} \quad
        \lambda >0,
        \] 
        which is only possible if $\q F(0_{m,m}) = 0_{n,n}$ and analogously, 
        $\q G(0_{n,n}) = 0_{m,m}$. Thus,
        \begin{equation}\label{eq:positive homogenuity}
            \q T(0_{m,m})
            =
            (1- \alpha) 0_{m,m} + \alpha \gamma_\q G \q G( \gamma_\q F \q F(0_{m,m})) 
            =  (1-\alpha) 0_{m,m} + \alpha \gamma_\q G \q G(0_{n,n})
            = 0_{m,m},
        \end{equation}
        which is the unique fixed point of $\q T$ by Theorem \ref{thm: convergence grad}.
       \end{remark}

       Alternatively, it is possible to consider ''parallel'' (Jacobi-like) updates by using just $B^t$ instead of $B^{t+1}$ in the second step of the algorithms.
However, in our numerical experiments this takes longer to converge and does not lead to improved results, 
so that we focus on ''sequential'' (Gauss-Seidel like) updates.

\subsection{Stochastic Fixed Point Iterations}\label{subsec:stoch_algs}  
In practical applications, the complete evaluation of $\q F(A)$ may become costly. Therefore, we consider a stochastic version 
of Algorithm \ref{alg:fix}. 
To this end, rewrite the steps of the algorithm as
\begin{align}
B^{t+1} &= B^t + \gamma_\q F \q F(A^t) - B^t,\\
A^{t+1} &= A^t + \alpha \left( \gamma_G \q G(B^{t+1}) - A^t \right).
\end{align}
The idea  is now to update only one entry in $A$ and $B$ in a step. 
More precisely, we consider two families of operators
$R^{(i,j)}: \R^{m \times m} \times \R^{n \times n} \to \R^{m \times m}$ with $i,j \in [n]$
and
    $\q S^{(k,\ell)}: \R^{m \times m} \times \R^{n \times n} \to \R^{n \times n}$ with $k,\ell \in [m]$ defined as
    \begin{equation}
    \begin{aligned}\label{eq: op stochastic}
    \q R^{(i,j)} (A, B)
        & \coloneqq B + ( \gamma_\q F \q F_{i,j}(A) - B_{i,j}) E^{i,j}_{n},\\
        \q S^{(k,\ell)} (A, B)
        & \coloneqq A + \alpha ( \gamma_\q G \q G_{k,\ell}(B) - A_{k,\ell}) E^{k,\ell}_{m},
    \end{aligned}
    \end{equation} 
where $E^{i,j}_{n}$ is $n \times n$ matrix with a single nonzero entry $(E^{i,j}_{n})_{i,j} = 1$. 
Then we consider the \emph{Random Function Iterations} (RFI) in Algorithm \ref{alg: rfi}. This algorithm was  originally proposed in \cite{hermer2020random}.
    
    \begin{algorithm}[h!]
        \caption{Random Function Iteration (RFI) for \eqref{fix1} }\label{alg: rfi}
        \vspace{1mm}
        \KwData{Initialization $A^0 \in \R^{m \times m}$, $B^0 \in \R^{n \times n}$, parameter $0<\alpha < 1$}
      Let $\xi_t = (k_t,\ell_t,i_t,j_t)$, 
        $t \in \N$ be the sequence of independent identically distributed random variables sampled uniformly 
        from $[m]^2 \times [n]^2$ \\
        \For{$t=0,1,\ldots$}{
            Update $B^{t+1} = \q R^{(i_t,j_t)} (A^t,B^t)$. \\
            Update $A^{t+1} = \q S^{(k_t,\ell_t)} (A^t,B^{t+1}) $
       
        }       
          \end{algorithm}

 For $\xi \coloneqq (k,\ell,i,j) \in [m]^2 \times [n]^2$,
we introduce the operator $\q T^{\xi}: \R^{m \times m} \times \R^{n \times n} \to \R^{m \times m} \times \R^{n \times n}$ 
by
    \begin{equation}\label{eq: Q stochastic}
        \q T^{\xi}(A,B) 
        \coloneqq \left( \q T^{\xi,1}(A,B), \q T^{\xi,2}(A,B) \right) 
        \coloneqq \left( \q S^{(k,\ell)}(A, \q R^{(i,j)} (A,B) ), \q R^{(i,j)}(A,B) \right)
    \end{equation}
    Then, the iteration in Algorithm \ref{alg: rfi} can be written as 
    \[
        (A^{t+1}, B^{t+1}) 
        = \q T^{\xi_t} (A^t, B^t).
    \]
    To analyze the algorithm, we equip the space $\R^{m \times m} \times \R^{n \times n}$ with the norm 
    \[
        \| (A,B) \|_\infty 
        \coloneqq 
        \max\{ \| A \|_\infty, \| B \|_\infty \}.
    \]
    The next lemma describes important properties of  $\q T^\xi$.   
        
    \begin{lemma}\label{l: properties stochastic}
    Let $\q F$ and $\q G$ be Lipschitz continuous with  constants $L_{\q F}$ and $L_{\q G}$
    and
    $0 < L_\q F <  \frac{1}{\gamma_{\q F}}$ 
    and $0 < L_\q G < \frac{1}{\gamma_{\q G}}$. 
    Then, 
    for all $\xi = (k,\ell ,i ,j) \in [m]^2 \times [n]^2$ 
    the operators $\q R^{(i,j)}$, $\q S^{(k,\ell)}$ 
    and $\q T^\xi$ 
    are nonexpansive. 
    Furthermore, 
    there exists an unique tuple
    $(A,B) \in \R^{m \times m} \times \R^{n \times n}$,
    such that  
    \begin{equation}\label{eq: fixed point sets}
        \bigcap_{\xi \in [m]^2 \times [n]^2} \Fix(\q T^\xi) 
        = \left\{ (A, \gamma_\q F \q F(A) ) 
        : \, A \in \Fix(\q T) \right\}
        = \{ (A,B) \}.
    \end{equation} 
    \end{lemma}
        
    \begin{proof}
        Let $\xi = (k,\ell, i,j) \in [m]^2 \times [n]^2$ and consider $A, A' \in \R^{m \times m}$ and $B, B' \in \R^{n \times n}$. 
    
        First, we show that $\q S^{(k,\ell)}$ is nonexpansive.  
        We consider two possible cases for indices $k',\ell' \in [m]$. 
        If $(k',\ell') \neq (k,\ell)$, 
        then $\q S^{(k,\ell)}_{k',\ell'}(A,B) = A_{k',\ell'}$ and
        \[
            | \q S^{(k,\ell)}_{k',\ell'}(A,B) -  \q S^{(k,\ell)}_{k',\ell'}(A',B') | 
            = | A_{k',\ell'} - A'_{k',\ell'} | 
            \le \| A - A' \|_\infty
            \le \| (A,B) - (A',B') \|_\infty.
        \]
        Otherwise, we obtain
        \[
            \q S^{(k,\ell)}_{k,\ell}(A,B) 
            =  (1- \alpha )A_{k,\ell} + \alpha \gamma_\q G \q G_{k,\ell}(B) 
        \]
        and
         \begin{align*}
            | \q S^{(k,\ell)}_{k,\ell}(A,B) - \q S^{(k,\ell)}_{k,\ell}(A',B') | 
            & = | (1- \alpha )A_{k,\ell} 
            + \alpha \gamma_\q G \q G_{k,\ell}(B) - [(1-\alpha )A'_{k,\ell} + \alpha \gamma_\q G \q G_{k,\ell}(B')] | \\
            & \leq (1- \alpha ) |A_{k,\ell} - A'_{k,\ell}|
            + \alpha \gamma_{\q G}|\q G_{k,\ell}(B) - \q G_{k,\ell}(B')|\\
            & < (1 - \alpha) \|A - A'\|_\infty
            + \tfrac{\alpha}{L_{\q G}} \|\q G(B) - \q G(B')\|_\infty\\
            & \leq (1 - \alpha) \| A - A' \|_\infty 
            + \alpha \| B - B' \|_\infty \\
            & \le \| (A,B) - (A',B') \|_\infty.
        \end{align*}
        In summary, we conclude
        \begin{equation}\label{eq: S_kl nonexpansive}
            \| \q S^{(k,\ell)}(A,B) - \q S^{(k,\ell)}(A',B') \|_\infty 
            \le 
            \| (A,B) - (A',B') \|_\infty.
        \end{equation}
        Similar derivations show that $\q R^{(i,j)}$ is nonexpansive.
        Consequently, 
        $\q T^{\xi,2} = \q R^{(i,j)}$ is nonexpansive. 
        Also $\q T^{\xi,1}$ is nonexpansive, since
        \begin{align}
            \| \q T^{\xi,1}(A,B) - \q T^{\xi,1}(A',B') \|_\infty
            & = \| \q S^{(k,\ell)}(A, \q R^{(i,j)} (A,B) ) - \q S^{(k,\ell)}(A', \q R^{(i,j)} (A',B') ) \|_\infty \\
            & \le \| (A, \q R^{(i,j)} (A,B)) - (A',\q R^{(i,j)} (A',B')) \|_\infty \\
            & = \max\{ \| A - A' \|_\infty, \| \q R^{(i,j)} (A,B) - \q R^{(i,j)} (A',B') \|_\infty \} \\
            & \le \max\{ \| A - A' \|_\infty, \| (A,B) - (A',B') \|_\infty \} \\
            & = \| (A,B) - (A',B') \|_\infty.
            \label{eq : Q1 nonepxansive}
        \end{align} 
        Utilizing \eqref{eq : Q1 nonepxansive} 
        and \eqref{eq: S_kl nonexpansive} we get finally
        \begin{align*}
            & \| \q T^{\xi}(A,B) - \q T^{\xi}(A',B') \|_\infty \\
            & = \max\left\{ \| \q T^{\xi,1}(A,B) - \q T^{\xi,1}(A',B') \|_\infty, \| \q T^{\xi,2}(A,B) - \q T^{\xi,2}(A',B') \|_\infty \right\} \\
            & \le \| (A,B) - (A',B') \|_\infty.
        \end{align*}
        To show that \eqref{eq: fixed point sets} holds, 
        we consider $(A,B) \in \bigcap_{\xi \in [m]^2 \times [n]^2} \Fix(\q T^{\xi})$. 
        Then,
        for every $\xi = (k,\ell, i,j) \in [m]^2 \times [n]^2$
        we have
        \[
            B_{i,j} = \q T^{\xi,2}_{i,j}(A,B) 
            = \q R^{(i,j)}_{i,j}(A,B)
            = B_{i,j} 
            + ( \gamma_\q F \q F_{i,j}(A) - B_{i,j}) 
            = \gamma_\q F \q F_{i,j}(A),
        \] 
        so that $B = \gamma_\q F \q F(A)$. 
        With $(A,B) \in \Fix(\q T^\xi)$, this yields
        $\q R^{(i,j)}(A,B) = B = \gamma_\q F \q F(A)$.
        By definition of $\q T$, we finally obtain
        \begin{align}
            A_{k,\ell} 
            & = \q T^{\xi,2}_{k,\ell}(A,B)
            = \q S^{(k,\ell)}_{k,\ell} (A, \q R^{(i,j)} (A,B) )
            = \q S^{(k,\ell)}_{k,\ell} (A,  \gamma_\q F \q F(A)) \\
            & = (1- \alpha )A_{k,\ell} + \alpha \gamma_\q G \q G_{k,\ell}(  \gamma_\q F \q F(A))
            = \q T_{k,\ell}(A)
            \label{eq: fix A under Qkl}
        \end{align}
       for all $(k,l) \in [m]^2$.
        Therefore, 
        $(A,B) \in \bigcap_{\xi \in [m]^2 \times [n]^2} \Fix(\q T^{\xi})$ implies $A \in \Fix(\q T)$ 
        and $B = \gamma_\q F \q F(A)$. 
        The reverse inclusion follows analogously.
        
        Finally, 
        we note that by \eqref{eq: fix A under Qkl} also holds
        $A = \gamma_\q G \q G(\gamma_\q F \q F(A)) = \gamma_\q G \q G(B)$.
        Hence,
        $(A,B)$ is a fixed point of operator $(\gamma_{\q G} \q G, \gamma_{\q F} \q F)$. Since $L_{\q F} < \frac{1}{\gamma_\q F}$ and $L_{\q G} < \frac{1}{\gamma_\q G}$, we know that this operator
        has a unique fixed point $(A,B)$ 
        by Banach's Fixed Point Theorem.
    \end{proof}

Unfortunately,  we cannot use the convergence results shown for RFI in \cite{hermer2019random}, since those were only established for paracontractive operators $\q T$, i.e. operators
fulfilling 
$$
\|\q T(A) - A' \|_\infty < \| A - A'\|_\infty
\quad \text{for all} \quad A \notin \Fix(\q T), \ A' \in \Fix(\q T).
$$
Our operator $\q T^{\xi}$ is nonexpansive, but not paracontractive.  
Unfortunately, establishing paracontractiveness for the $\ell_\infty$-norm is not possible due to the following considerations: 
consider $A^t$ that has two entries $(k,\ell)$ and $(k', \ell')$ 
such that 
\[
    |A^t_{k,\ell} - A_{k,\ell}| 
    = |A^t_{k',\ell'} - A_{k',\ell'}| 
    = \|A^t - A \|_\infty.
\]
Due to the single entry updates performed by $\q T^\xi$, 
even if $(k_t, \ell_t) = (k',\ell')$ decreases the error 
$|A^{t+1}_{k',\ell'} - A_{k',\ell'}| < |A^t_{k',\ell'} - A_{k',\ell'}|$, 
we still have $A^{t+1}_{k,\ell} = A^t_{k,\ell}$ 
and consequently $\| A^{t+1} - A\|_\infty = \|A^t - A \|_\infty$. 
The extreme scenario when all entries of $A^t - A$ have the same absolute value, shows that a unless (nonstochastic) $\q T$ is used, the error will remain unchanged.  	

Nevertheless, for convenience, we like to mention that
it is possible to obtain the following weaker convergence guarantees by combining Lemma \ref{l: properties stochastic} and \cite[Thm.~2.17]{hermer2020random}.
    
    \begin{theorem}    \label{thm: convergence rfi}
        Let $\q F$ and $\q G$ be Lipschitz continuous with  constants $L_{\q F}$ and $L_{\q G}$, 
        and  $0 < L_\q F < \frac{1}{\gamma_\q F}$ 
        and $0 < L_\q G < \frac{1}{\gamma_\q G}$. 
        Let $( (A^t, B^t) )_{t \in \N}$ be a sequence
        of random variables 
        generated by Algorithm \ref{alg: rfi}, 
        and $\q P^t$ the probability measure 
        corresponding to the law of $(A^t, B^t)$. 
        Then,
        $\tfrac{1}{T} \sum_{t=0}^{T-1} \q P^t$ 
        converges to a point measure $\delta_{(A,B)}$ 
        in Prokhorov-Levy metric as $T \to \infty$, 
        where $(A,B)$ is the unique fixed point 
        in \eqref{fix1}.
   \end{theorem}
   
    To obtain stronger convergence guarantees, 
    we instead consider convergence in the $\ell_2$-norm.
    To this end, we equip the space $\R^{m \times m} \times \R^{n \times n}$ with the norm 
    \[
        \| (A,B) \|_2
        \coloneqq 
         \big( \| A \|_2^2 + \| B \|_2^2 \big)^\frac12.
    \]
    Indeed, by the next theorem, we get linear convergence rate of RFI. 	
    
    \begin{theorem}\label{thm: convergence stoch}
    Let $\q F$ and $\q G$ be Lipschitz continuous with  constants $L_{\q F}$ and $L_{\q G}$, and
    $0 < L_\q F \le  \frac{\sqrt \alpha}{\sqrt 2 m \gamma_{\q F}}$
    and 
    $0 < L_\q G \le \frac{1}{n \gamma_{\q G}}$, where $n,m > 1$. 
    Let $(A,B)$ denote the unique fixed point of $(\gamma_\q G \q G, \gamma_\q F \q F)$. 
    Then the sequence $( (A^t, B^t) )_{t \in \N}$ generated by Algorithm \ref{alg: rfi}  converges almost surely to $(A,B)$ 
    and fulfills
    \[
        \bb E_{\xi_t} \left[ \| (A^{t+1}, B^{t+1}) - (A,B)\|_2^2 \right]
        \le 
        L \| (A^t, B^t) - (A,B) \|_2^2
        \quad \text{for all} \quad   t \in \N,   
    \]	
    where $\bb E_{\xi_t}$ denotes the expectation with respect to $\xi_t$ 
    and
    \[
        L
        \coloneqq
        \max \left\{ 1 - \frac{\alpha}{m^2} + (1 + \alpha \gamma^2_\q G L^2_\q G ) \gamma^2_\q F L^2_\q F  , (1 + \alpha \gamma^2_\q G L^2_\q G) \left(1 - \frac{1}{n^2} \right) \right\} < 1.
    \]
    \end{theorem}
    
    \begin{proof}
        Let $\xi_t = (k_t, \ell_t, i_t, j_t) \in [m]^2 \times [n]^2$ 
        be as in Algorithm \ref{alg: rfi}. 
        Then, 
        by construction, 
        \begin{align*}
            \| A^{t+1} - A \|_2^2 
            & = \sum_{k,\ell = 1}^{m} 
            | \q S^{(k_t,\ell_t)}_{k,\ell}(A^t,B^{t+1}) - A_{k,\ell}|^2 \\
            & =  
            \sum_{\substack{k,\ell = 1\\ (k,\ell) \neq (k_t, \ell_t)}}^{m} 
            | A^t_{k,\ell} - A_{k,\ell}|^2
            + |(1 - \alpha) A^t_{k_t,\ell_t} 
            + \alpha \gamma_{\q G} \q G_{k_t,\ell_t}(B^{t+1}) - A_{k_t, \ell_t} |^2 \\
            & = \| A^t - A \|_2^2 + |(1- \alpha ) A^t_{k_t,\ell_t} + \alpha \gamma_\q G \q G_{k_t,\ell_t}(B^{t+1}) - A_{k_t, \ell_t} |^2 - | A^t_{k_t,\ell_t} - A_{k_t,\ell_t}|^2.
        \end{align*} 
        Since $B^{t+1}$ only depends on $i_t, j_t$, we can take the expectation with respect to $k_t$ and $\ell_t$ yielding 
        \begin{align*}
            \bb E_{k_t,\ell_t} \| A^{t+1} - A \|_2^2 
            & = \| A^t - A \|_2^2 + \frac{1}{m^2} \| (1- \alpha ) A^t + \alpha \gamma_\q G \q G(B^{t+1}) - A \|_2^2 - \frac{1}{m^2} \| A^t - A \|_2^2.
        \end{align*}
        Now, we transform the second term. Since $(A,B)$ is a fixed point of 
        $(\gamma_\q G \q G, \gamma_\q F \q F)$, we can write $A = \gamma_\q G \q G(B)$. This gives
        \[
            \| (1- \alpha ) A^t  + \alpha \gamma_\q G \q G(B^{t+1}) - A \|_2^2 
            = 
            \| (1- \alpha ) (A^t - A) 
            + \alpha \left(\gamma_\q G \q G(B^{t+1}) - \gamma_\q G \q G(B) \right) \|_2^2.
        \]
        Using the convexity of $\| \cdot \|_2^2$, 
        we bound
        \[
            \| (1- \alpha ) A^t  + \alpha \gamma_\q G \q G(B^{t+1}) - A \|_2^2
            \le 
            (1 - \alpha) \| A^t - A \|_2^2 
            + \alpha \| \gamma_\q G \q G(B^{t+1}) - \gamma_\q G \q G(B) \|_2^2.
        \] 
        Next, we use that $\q G$ is $L_\q G$-Lischitz continuous
        and the inequalities 
        \[
            \| A \|_2^2 \le m^2 \| A \|_\infty^2 
            \quad \text{and} \quad 
            \| B \|_\infty \le \| B \|_2,
            \quad \text{for all} 
            \quad 
            A \in \R^{m \times m}, 
            B \in \R^{n \times n},
        \]
        to obtain
        \begin{align*}
            \| (1- \alpha ) A^t  + \alpha \gamma_\q G \q G(B^{t+1}) - A \|_2^2
            & \le (1- \alpha ) \| A^t - A \|_2^2 
            + \alpha m^2 \gamma^2_\q G \| \q G(B^{t+1}) - \q G(B) \|_\infty^2 \\
            & \le (1- \alpha ) \| A^t - A \|_2^2 
            + \alpha m^2 \gamma^2_\q G L^2_\q G \| B^{t+1} - B \|_\infty^2 \\
            & \le (1- \alpha ) \| A^t - A \|_2^2 
            + \alpha m^2 \gamma^2_\q G L^2_\q G \| B^{t+1} - B \|_2^2.
        \end{align*}
        This yields 
        \begin{equation}\label{eq: Ekl}
            \bb E_{k_t,\ell_t} \left[\| A^{t+1} - A \|_2^2 \right]
            \leq 
            \left( 1 - \frac{\alpha}{m^2} \right) \| A^t - A \|_2^2
            + \alpha \gamma^2_\q G L^2_\q G \| B^{t+1} - B \|_2^2.
        \end{equation}
        Repeating these steps for $ \bb E_{i_t,j_t} \left[\| B^{t+1} - B \|_2^2\right]$
        provides
        \begin{equation}\label{eq: Eij}
            \bb E_{i_t, j_t} \left[ \| B^{t+1} - B \|_2^2 \right]
            \le 
            \left( 1 - \frac{1}{n^2} \right) \| B^t - B \|_2^2 
            + \gamma^2_\q F L^2_\q F \| A^t - A \|_2^2
        \end{equation}
        for all $t \in \N$.
        Combining \eqref{eq: Ekl} and \eqref{eq: Eij} yields
        \begin{align*}
            & \bb E_{\xi_t} \| (A^{t+1},B^{t+1}) -(A,B) \|_2^2
            =  \bb E_{i_t, j_t} \left[ \bb E_{k_t,\ell_t} \| A^{t+1} - A \|_2^2 \right] 
            + \bb E_{i_t, j_t} \| B^{t+1} - B \|_2^2 \\
            & \qquad 
            \le \left( 1 - \frac{\alpha}{m^2} \right) \| A^t - A \|_2^2 
            + (1 + \alpha \gamma^2_\q G L^2_\q G ) \bb E_{i_t, j_t} \| B^{t+1} - B \|_2^2 \\
            & \qquad 
            \le \left( 1 - \frac{\alpha}{m^2} +  (1 + \alpha \gamma^2_\q G L^2_\q G ) \gamma^2_\q F L^2_\q F  \right) \| A^t - A \|_2^2
            + (1 + \alpha \gamma^2_\q G L^2_\q G) \left(1 - \frac{1}{n^2} \right) \| B^t - B \|_2^2 \\
            & \qquad 
            \le L \| (A^{t},B^{t}) -(A,B) \|_2^2
        \end{align*}
        for all $t \in \N$.
        Finally, 
        we show that the constant $L$ is less than $1$ 
        by the choice of $\gamma_\q F$ and $\gamma_\q G$. 
        We have
        \[
        (1 + \alpha \gamma^2_\q G L^2_\q G) \left(1 - \frac{1}{n^2} \right)
        \le \left(1 + \frac{1}{n^2} \right) \left(1 - \frac{1}{n^2} \right)
        = 1 - \frac{1}{n^4} < 1,
        \]
        and
        \begin{align*}
            1 - \frac{\alpha}{m^2} + \left(1 + \alpha \gamma^2_\q G L^2_\q G \right) \gamma^2_\q F L^2_\q F
            \le 1 - \frac{\alpha}{m^2} + \left(1 + \frac{1}{n^2}\right) \frac{\alpha}{2 m^2}
            = 1 - \left(1 - \frac{1}{n^2}\right) \frac{\alpha}{2 m^2}
            < 1.
        \end{align*}	        
        From these inequality, we conclude that
        the sequence $Y_t \coloneqq \| (A^{t},B^{t}) -(A,B) \|_2^2$ 
        is a nonnegative supermartingale.
        Thus, 
        by \cite[Cor.\ 13.3.3]{athreya06measure}, 
        it converges almost surely to a finite random variable $Y$.
        It is nonnegative as limit of a nonnegative sequence 
        and Fatou's lemma \cite[Thm.~12.2.2]{athreya06measure} 
        implies that 
        \[
            0 \le \bb E \left[ Y \right]
            = \lim_{t \to \infty} \bb E \left[ Y_t \right]
            \le \bb E \left[Y_0 \right] \, \lim_{t \to \infty} L^t = 0. 
        \]  
        Since $Y$ is nonnegative, 
        it implies that $Y = 0$ almost surely. 
        Thus, 
        $\| (A^t, B^t) - (A,B) \|_2 \to 0$ 
        and $(A^t, B^t) \to (A,B)$ almost surely as $t \to \infty$.
    \end{proof}

    The inequality for expectation obtained in Theorem \ref{thm: convergence stoch} states that $\q T^\xi$ is paracontractive on average, which is sufficient for convergence. This is weaker than the notion of almost-firmly nonexpansiveness on average used for convergence in \cite{hermer2020random}.  
    
    \begin{remark}
       We could also consider joint stochastic updates given by
        \[
       (A^{t+1},B^{t+1}) \coloneqq \left( \q S^{(k,\ell)}(A^t, B^t), \q R^{(i,j)}(A^t,B^t) \right).
        \]
        Under the same assumptions, an analogy of Theorem \ref{thm: convergence stoch} can be proved with a better convergence rate  
        \[
        L\coloneqq  \max \left\{ 1 - \tfrac{\alpha}{m^2} + \alpha \gamma^2_\q F L^2_\q F , 1 - \tfrac{\alpha}{n^2} + \alpha \gamma^2_\q G L^2_\q G\right\},
        \]
        which scales as $1 - \max^{-2}\{ n,m \}$ in comparison to $L \approx 1 - \max^{-2}\{ n^2, m \}$ in Theorem~\ref{thm: convergence stoch}.   
    \end{remark}

    \section{Optimal Transport Distances} \label{sec: wasserstein}
   In the next three sections, we want to apply our findings from the previous section to special functions $\q F$ and $\q G$
   depending on given unlabeled data.
   For samples $X_i \in \R^m_{\ge 0}$, $i = 1,\ldots, n$, let
$$
X = (X_1 \, \ldots \, X_n) =  
\begin{pmatrix}
    (X^1)^\tT\\
    \vdots
    \\
    (X^m)^\tT
\end{pmatrix} \in \mathbb R^{m \times n},
$$
where  $X^k \in \R^m$, $k = 1,\ldots, m$ characterizes the $k$-th feature of the samples. We assume that $X_i \not = X_j$ for $i \not = j$ and $X^k \not = X^\ell$ for $k \not = \ell$, which
can be simply achieved by canceling repeating columns and rows in $X \in \R^{m \times n}$.
We also require normalization of $X$, which is achieved by considering two copies of the dataset, $\uX$ with normalized rows and $\overline{X}$ with normalized columns. Then,
\begin{equation}\label{normalization}
\uX \1_n = \1_m \quad \text{and} \quad \overline{X}^\tT \1_m= \1_n. 
\end{equation}

\subsection{Optimal Transport Distances From Metric Matrices}
     In this section, we are interested in \emph{metric matrices}
     \begin{equation}\label{eq: metric matrices}
        \bb D_m 
        \coloneqq
        \{ A \in \bb R^{m \times m}_{\ge 0} : 
        A = A^\tT, 
        A_{k,k} = 0,
       A_{k,\ell} > 0, k \neq \ell,
       \text{ and } 
       \\
        A_{k,s} \le A_{k,\ell} + A_{\ell,s}
      \},
    \end{equation}
    and their closure, the pseudo-distance matrices
    \[
        \overline{\bb D}_m 
        \coloneqq
        \{ A \in \bb R^{m \times m}_{\ge 0} : 
        A = A^T, 
        A_{k,k} = 0, 
        \text{ and } A_{k,s} \le A_{k,\ell} + A_{\ell,s},
      \}.
    \]
    To emphasize the dimensions, we use the abbreviations
    $$
\mathbb A \coloneqq \bb D_m \quad \text{and} \quad 
\mathbb B \coloneqq \bb D_n.
$$

For $A \in \overline{\mathbb A}$ and $x,y \in \R^m_{\ge 0}$ with $x^\tT \1_m = y^\tT \1_m = 1$,
the \emph{optimal transport} (OT) \emph{pseudo-distance} is given by 
    \begin{equation}\label{eq: optimal transport}
    W_A(x, y) \coloneqq \min_{P \in \Pi(x,y)} \ \langle A, P \rangle,
    \end{equation}
    where
    $$
     \Pi(x,y) \coloneqq \{
    P \in \R^{m \times m}_{\ge 0}, \ 
    P \1_m = x, \; 
    P^\tT \1_m =y\} .   
    $$    
Indeed $W_A$ becomes a distance for $A \in \mathbb A$.
We are interested in functions
$\mathcal F: \overline{\mathbb A}  \to \R^{n \times n}$ 
and
$\mathcal G:\overline{\mathbb B}  \to \R^{m \times m}$ of the form
\begin{align}\label{eq:kernelregularized_w}
\mathcal F(A) \coloneqq W_A (\uX) + R_n 
\quad \text{and} \quad 
\mathcal G(B) \coloneqq  W_B (\overline{X}) + R_m ,
\end{align}
where $R_m \in \overline{\mathbb A}$ and $R_n \in  \overline{\mathbb B}$ are non-zero matrices, and
    \begin{align} \label{eq:wasser}
	W_A (\uX) \coloneqq \Big(  W_A(\uX_i,\uX_j)  \Big)_{i,j=1}^n
  \quad \text{and} \quad
    W_B (\overline{X}) \coloneqq \Big( W_B (\overline{X}^k, \overline{X}^{\ell}) \, \Big)_{k,\ell=1}^m.
	\end{align} 
By definition of the Wasserstein distance, it follows immediately that $\uX \mapsto W_A (\uX)$ is positively homogeneous as a function in $A$ and inclusion of $R_n$ in \eqref{eq:kernelregularized_w} counteracts issues with positive homogenuity described by Remark \ref{rem:pos_def}.

\begin{remark} \label{peyre}   
In \cite{huizing2022unsupervised}, the authors were interested in the mappings
    \begin{equation}
        \begin{aligned}\label{eq: wasserstein peyre}
            A \mapsto W_A(\uX) + \| A \|_\infty R_n 
            \quad \text{and} \quad
            B \mapsto W_B(\overline X) + \| B \|_\infty R_m ,
        \end{aligned}
    \end{equation}
    which are positively homogeneous and therefore, by  Remark \ref{rem:pos_def}, are not interesting with respect to 
    Algorithm \ref{alg:fix}.
    \\
   However, these maps are considered in the context of  Algorithm \ref{alg:fix2}, in which case the iterates fulfill $\|A^t\|_\infty = \|B^t\|_\infty = 1$ for all
    $t \in \N$. 
   In this case, \eqref{eq: wasserstein peyre} coincides with our setting \eqref{eq:kernelregularized_w}.
    \end{remark}

The following lemma summarizes properties of $\q F$ which hold similarly for $\q G$. Besides the Lipschitz property, it is important that $\q F$ maps into the correct domain $\mathbb B$, resp., its closure.

\begin{lemma}\label{prop: wasserstein properties}
        The function $\q F$  defined by \eqref{eq:kernelregularized_w} has the following properties: 
        \begin{enumerate}
            \item $\q F$ is 1-Lipschitz on $\overline{\bb A}$;
            \item $\q F$ maps $\overline{\bb A}$ to $\{ B \in \overline{\bb B}: B_{i,j} \ge (R_n)_{i,j} \ge 0, i \neq j \}$ with strict last inequality
            and $B \in \mathbb B$  if $R_n \in \mathbb B$. 
            \item $\| \q F(A) \|_\infty \ge \| R_n \|_\infty$ 
            for all $A \in \overline{\bb A}$.           
        \end{enumerate}
    \end{lemma}
    
    \begin{proof}
        1. For  $A,A' \in \overline{\mathbb A}$, we have
        \begin{align}
            \| \q F(A) - \q F(A) \|_\infty &= \| W_A(\uX) - W_{A'}(\uX) \|_\infty \\
            &= \max_{i,j\in [n]} \big| \min_{P \in \Pi(\uX_i,\uX_j)} \langle A,P \rangle - \min_{P' \in \Pi(\uX_i,\uX_j)} \langle A',P' \rangle \big|.
         \end{align}
         Let $P^*_{i,j}$ be the minimizer of the smaller minimum for each $i,j \in [n]$. Then we obtain
         \begin{align}
            \| \q F(A) - \q F(A) \|_\infty 
            &\le \max_{i,j\in [n]} \langle A,P^*_{i,j} \rangle - \langle A',P^*_{i,j} \rangle            
            \le
            \|A - A'\|_\infty  \max_{i,j\in [n]} \|P^*_{i,j}\|_1 = \|A - A'\|_\infty.    
        \end{align}                       
        2. For $A \in \overline{\mathbb A}$, we know that $W_A$ 
        is a pseudometric and since $R_n \in \overline{\mathbb B}$, we get
        \[
            \q F_{i,j}(A) 
            = W_A(\uX_i, \uX_j) + R_n(\uX_i, \uX_j) 
            \ge R_n(\uX_i, \uX_j) 
            \ge 0
        \]
        with strict last inequality if $R_n \in \mathbb B$.
        
        3. By definition of $\q F$,
        we obtain
        \[
            \| \q F(A) \|_\infty 
            = \max_{i,j \in [n]} \big( W_A(\uX_i, \uX_j) + R_n(\uX_i, \uX_j) \big)
            \ge 
            \max_{i,j \in [n]} R_n(\uX_i, \uX_j) 
            = \| R_n \|_\infty.
        \]   
    \end{proof}
Using the above properties, the following convergence guarantees 
follow directly from  the Theorems  \ref{l:lipshcitz}, \ref{thm: convergence grad}
and \ref{thm: convergence stoch}.
    
    \begin{theorem}  \label{cor: wasserstein}
    Let $\q F$ and $\q G$ be defined by \eqref{eq:kernelregularized_w}.
    Then, starting with $A^0 \in \overline{\bb A}$, $B^0 \in \overline{\bb B}$, the following holds true:
        \begin{enumerate}
            \item if $\|R_n\|_\infty, \|R_m\|_\infty > 2$, then
            the sequence $(A^t,B^t)_t$ generated by Algorithm \ref{alg:fix2}
            converges to the unique fixed point  $(A, B) \in \overline{\bb A} \times \overline{\bb B}$ of \eqref{fix2}.
            \item if $0 < \gamma_\q F, \gamma_\q G$ and $\gamma_\q F \gamma_\q G < 1$, 
            then the sequence $(A^t,B^t)_t$ generated by Algorithm \ref{alg:fix}
            converges to a unique fixed point
            $(A, B) \in \overline{\bb A} \times \overline{\bb B}$ 
            of \eqref{fix1};
            \item if $0 < \gamma_\q F \le \sqrt \alpha/ \sqrt 2 m$, $0< \gamma_\q G \le 1/n$, $n, m \ge 2$, then the sequence $(A^t,B^t)_t$ generated by Algorithm \ref{alg: rfi} 
            converges almost surely
            to the unique fixed point $(A, B) \in \overline{\bb A} \times \overline{\bb B}$ of \eqref{fix2}.
                   \end{enumerate} 
        If additionally $R_m \in \bb A$, $R_n \in \bb 
 B$,
        then $(A,B) \in \bb A \times \bb B$.
    \end{theorem}

    \begin{remark}\label{peyre_1}
    For OT distances, Algorithm \ref{alg: rfi} resembles the stochastic power iteration algorithm  
    introduced in \cite{huizing2022unsupervised} for problem \eqref{fix2}. 
    In contrast to those algorithms, we do not require normalization of $A^t$ after each iteration,
    nor scaling factors $\tilde \mu, \tilde \nu$ 
    approximating $\|\q F(A)\|_\infty$ and $\|\q G(B)\|_\infty$ of the fixed point $(A,B)$. 
    Furthermore, 
    our convergence guarantees hold for constant step sizes, 
    while for the stochastic power iterations, the step sizes vanish over time. 

    Further, we were not able to follow few steps in the convergence proof of the stochastic power iteration algorithm in \cite[Appendix E]{huizing2022unsupervised}:  we do not understand how the projection theorem is applied for the operation $A/\|A\|_\infty$ which is not a projection onto the $\|\cdot\|_\infty$-ball and how Lipschitz continuity of $W_B$ in the $\|\cdot\|_2$-norm can be used without affecting the constants, while only Lipschitz continuity in the $\|\cdot\|_\infty$-norm was established for the claim.
     \end{remark}
   
\subsection{Sinkhorn Divergences} \label{sec:sinkhorn}    
Since the evaluation of Wasserstein distances $W_A(X_i, X_j)$, $i,j \in [n]$, is computationally heavy for large $m,n$, the Wasserstein distance is usually replaced by the
Sinkhorn divergence \cite{feydy2019interpolating} which does not satisfy
a triangular inequality.
Therefore, we may relax the assumptions on $A$ and $B$
to semi-distance matrices
\[
        \bb{SD}_m 
        \coloneqq 
        \{ A \in \bb R^{m \times m}_{\ge 0} : A = A^\tT, A_{k,k} = 0, A_{k,\ell} > 0, 
        k \neq \ell,
        k,\ell \in [m]\},
    \]
    and their closure 
    \[
        \overline{\bb{SD}}_m
        \coloneqq 
        \{ A \in \bb R^{m \times m}_{\ge 0} : 
        A = A^\tT, 
        A_{k,k} = 0, 
        k \in [m]\}
    \]
    and consider in this section
$$\bb A \coloneqq \bb{SD}_m \quad \text{and} \quad 
        \bb B \coloneqq \bb{SD}_n.$$
For $A \in \overline{\mathbb A}$ and $x,y \in \R^m_{\ge 0}$ with $x^\tT \1_m = y^\tT \1_m = 1$ and $\varepsilon > 0$, let
    \begin{align}\label{eq: entropic ot}
    W_A^\varepsilon (x, y) 
    &\coloneqq 
    \min_{P \in \Pi(x,y)} \left\{ \langle A, P \rangle + \varepsilon \|A\|_\infty \KL( P, x y^\tT) \right\} \\
    &=
    \varepsilon \|A\|_\infty \min_{P \in \Pi(x,y)} \KL \left( P, x y^\tT \circ \exp\big(-\tfrac{1}{\varepsilon \|A\|_\infty} A \big) \right)
    \quad \text{for} \quad \|A\|_\infty > 0,
    \end{align}
    where $\KL$ is the \emph{Kullback-Leibler divergence}
    defined for $P,Q \in \R^{m \times m}_{\ge 0}$ with 
    by $\sum_{k,\ell = 1}^m P_{k,\ell} =\sum_{k,\ell = 1}^m Q_{k,\ell} = 1$ by
    \[
        \KL(P, Q) 
        \coloneqq
        \sum_{k,\ell=1}^{m} P_{k,\ell} \log\Big(\frac{P_{k,\ell}}{Q_{k,\ell}} \Big) \ge 0
    \]
    with the convention that $0 \log 0 \coloneqq 0$
    and $\KL(P, Q) = +\infty$ if $Q_{k,\ell} = 0$, but $P_{k,\ell} \not = 0$ for some $k,\ell \in [m]$.
    As $\KL$ is strictly convex, the minimizer exists, is unique and can be efficiently computed using Sinkhorn's algorithm \cite{cuturi2013sinkhorn}.
    Now, the \emph{Sinkhorn divergence} is defined by
    \[
        S^\varepsilon_A(x, y) 
        \coloneqq W^\varepsilon_A(x, y) 
        - \tfrac{1}{2}\big(W^\varepsilon_A(x, x) 
        + W^\varepsilon_A(y, y) \big), 
    \]
    As the KL divergence, the Sinkhorn divergence 
    admits $S^\varepsilon_A(x, y) \ge 0$
    with equality if $x = y$. The reverse implication is also true for $A \in \bb D_m$ \cite{feydy2019interpolating}.
    Note that we included $\|A\|_\infty$ as multiplier in \eqref{eq: entropic ot} 
    to make both $W^\varepsilon_A$ 
    and $S^\varepsilon_A$ positively homogeneous 
    with respect to $A$. 
   Instead of \eqref{eq:kernelregularized_w}, we deal with the  mappings 
   $\mathcal F: \overline{\mathbb A}  \to \R^{n \times n}$ 
and
$\mathcal G:\overline{\mathbb B}  \to \R^{m \times m}$ of the form
       \begin{equation}      \label{eq: sinkhorn regularized}
            \q F(A) 
             \coloneqq S^\varepsilon_A(\uX) +
             R_n
             \quad \text{and} \quad 
            \q G(B) 
             \coloneqq S^\varepsilon_B(\overline X)  +  R_m,
       \end{equation} 
    where where $R_m \in \overline{\mathbb A}$ and $R_n \in  \overline{\mathbb B}$ are non-zero matrices.

    The properties of these mappings are summarized by the next lemma.     
       
    \begin{lemma}    \label{prop: sinkhorn properties}
     The mapping $\q F$ defined in \eqref{eq: sinkhorn regularized} has the following properties:
        \begin{enumerate}
            \item $\q F$ is Lipschitz continuous on $\overline{\bb A}$ with constant $L_{\q F} \coloneqq 2(1 + \varepsilon C)$, where
            \[
            C \coloneqq  2 m \max_{i \in [n]} \sum_{\substack{k =1\\ X_{i,k} > 0}} ^m - \log (X_{i,k} ).
            \]
            \item $\q F$ maps 
            $\overline{\bb A}$ 
            to $\{ B \in \overline{\bb B} : B_{i,j} \ge (R_n)_{i,j} \ge 0, i \neq j \}$,
            with strict last inequality and $B \in \mathbb B$ if $R_n \in \mathbb B$.
            \item $\q F$  admits $\| \q F(A) \|_\infty \ge \| R_n \|_\infty$ for all $A \in \overline{\bb A}$.
        \end{enumerate}
    \end{lemma}
    
    \begin{proof}
        1. Let $i,j \in [n]$ be arbitrary fixed. Without loss of generality, assume that $W_A^{\varepsilon}(\uX_i, \uX_j) \ge W_{A'}^{\varepsilon}(\uX_i, \uX_j)$.
        Let $P, P'$
        be optimal plans in \eqref{eq: entropic ot}
        corresponding to $A$ and $A'$.
        Then we get
        \begin{align*}
            W_A^{\varepsilon}(\uX_i, \uX_j) 
            &= \langle A, P \rangle 
            + \varepsilon \|A\|_\infty \KL( P, \uX_i \uX_j^\tT)
            \le \langle A, P' \rangle
            + \varepsilon \|A\|_\infty \KL( P', \uX_i \uX_j^\tT).
        \end{align*}        
        and further
        \begin{align*}
            W_A^\varepsilon(\uX_i, \uX_j) - W_{A'}^\varepsilon(\uX_i, \uX_j)
            &\le \langle A, P' \rangle - \langle A', P' \rangle + \varepsilon ( \|A\|_\infty - \|A'\|_\infty) \KL( P', \uX_i \uX_j^\tT)\\
            & \le 
            \langle A - A' , P' \rangle + \varepsilon ( \|A  - A'\|_\infty) \KL( P', \uX_i \uX_j^\tT) \\
            & \le \big( \| P'\|_{1} + \varepsilon \KL( P', \uX_i \uX_j^\tT) \big) \| A - A' \|_\infty  \\
            & = \big( 1 + \varepsilon \KL( P', \uX_i \uX_j^\tT) \big) \| A - A' \|_\infty .
        \end{align*}
         Since we know by definition of $\KL$ that $P_{k,\ell}' =0$ whenever $\uX_{i,k} \uX_{j,\ell} = 0$ and we have 
        \[
            \KL(P', \uX_i \uX_j^\tT) = \smashoperator{\sum_{
            \substack{k,\ell=1\\
            \uX_{i,k} \uX_{j,\ell} >0}}}^{m} P_{k,\ell}' \log \left( \frac{P_{k,\ell}'}{\uX_{i,k} \uX_{j,\ell}} \right) .
        \]
        Let us consider the function $t \log(ct)$ for $t \in [0,1]$ with $c\ge1$. Its maximum is obtained at $t = 1$ with value $\log(c)$. Setting $t = P_{k,\ell}'$ and $c^{-1} = \uX_{i,k} \uX_{j,\ell}$ we obtain
        \begin{align*}
            \KL(P', \uX_i \uX_j^\tT) 
            & \le - \smashoperator{\sum_{\substack{k,\ell=1\\
            \uX_{i,k} \uX_{j,\ell} >0}}}^{m} \log \left( \uX_{i,k} \uX_{j,\ell} \right) 
              \le - m \smashoperator{\sum_{\substack{k=1\\
             \uX_{i,k} > 0}}}^{m} \log ( \uX_{i,k} ) 
            - m \smashoperator{\sum_{\substack{\ell=1\\\uX_{j,\ell} >0}}}^{m} \log( \uX_{j,\ell} )  \\
            & \le 2m \max_{i \in [n]} \smashoperator{\sum_{\substack{k=1\\
             \uX_{i,k} > 0}}}^{m} -\log ( \uX_{i,k} )  = C.
        \end{align*}
        Hence we conclude
        \[
        | W_A^\varepsilon(\uX_i,\uX_j) - W_{A'}^\varepsilon(\uX_i,\uX_j)|
        \le \left( 1 + \varepsilon C \right) \| A - A' \|_\infty,
        \]
        so that $W_A^\varepsilon(X_i,X_j)$ 
        is $(1 + \varepsilon C)$-Lipschitz in $A$. 
        Then, by definition,
        the Sinkhorn divergence $S^\varepsilon_A$ 
        is $(2 + 2 \varepsilon C)$-Lipschitz.
        Consequently, 
        $\q F$ is $(2 + 2 \varepsilon C)$-Lipschitz continuous
        in $\| \cdot \|_\infty$.  
        
        The rest of the proof follows the lines of the proof of 
        Lemma~\ref{prop: wasserstein properties}.
    \end{proof}

    Notably, there is a mismatch factor 2 between the 
    Lipschitz constants for the Wasserstein ($\varepsilon=0$) 
    and Sinkhorn ($\varepsilon>0$) case.
    We believe that a better Lipschitz constant $1+\varepsilon \tilde C$ could be achieved, but we could not prove this so far. 
    
   Using the above lemma,  Theorem \ref {cor: wasserstein} can be similarly deduced for the Sinkhorn case,
   so that we also have convergence of Algorithms 1-3 under some conditions.
   In particular,  
Theorem \ref {cor: wasserstein} requires lower bounds on $R_m$ and $R_n$ for the convergence of the sequence generated by Algorithm \ref{alg:fix2}.
The following theorem shows the existence of a
   fixed point of $\tilde {\q T} = \tilde {\q G} \circ \tilde {\q F}$ in \eqref{fix2} for \emph{arbitrary} semi-metric matrices $R_n$ and $R_m$.  For reasons outlined in  Remark \ref{rem:peyre_2}, we also include the proof of the theorem.

       \begin{theorem}\label{thm: existence of wasserstein singular}
      Let $\q F$ and $\q G$ be given by \eqref{eq: sinkhorn regularized}
      with $R_m \in \mathbb A$ and $R_n \in \mathbb B$.
        Then, 
        there exists a fixed point of $\tilde {\q T} = \tilde {\q G} \circ \tilde {\q F}$. The assertion also holds true
        for the Wasserstein setting, i.e. $\q F$ and $\q G$ in \eqref{eq:kernelregularized_w} and metric matrices $R_m,R_n$.
    \end{theorem}   

    \begin{proof}
    By Lemma \ref{prop: wasserstein properties}, we have
    for any $B \in \overline{\bb B}$ 
    that
    \[
        \q G(\overline{\bb B}) 
        \subset 
        \{ A \in \bb A: A_{k,\ell} \ge (R_m)_{k,\ell} > 0, k \neq \ell \}
    \]
    and $\|\q G(B)\|_\infty \ge \| R_m \|_\infty  > 0$. 
    Therefore,
    $\bb A$ being a cone implies $\q G(B) / \|\q G(B)\|_\infty \in \bb A$. 
    Next,
    we show that the entries of $\q G(B) / \|\q G(B)\|_\infty$ are bounded away from zero. 
    For this, we have by definition of $\tilde{\q F}$
    only to deal with $\|B\|_\infty =1$.
    Then the Lipschitz continuity of $\q G$ in Lemma \ref{prop: sinkhorn properties} yields
    \begin{align}
        \| \q G(B) \|_\infty 
        & \le \| \q G(B)  - \q G(0_{n,n} ) \|_\infty + \| \q G(0_{n,n} ) \|_\infty \\
        & \le (2 + 2 \varepsilon C) \| B - 0_{n,n} \|_\infty + \| \q G( 0_{n,n} ) \|_\infty
        = 2 + 2 \varepsilon C + \| \q G(0_{n,n} ) \|_\infty.
        \label{eq: bound G}
    \end{align}
    Since for all $i,j \in [n]$ it holds
    \begin{align*}
        0 \le W_{0_{n,n}}^\varepsilon(\uX_i, \uX_j)
        = 
        \min_{ P \in \Pi(\uX_i,\uX_j)} \varepsilon \|A\|_\infty \KL( P, \uX_i \uX_j^\tT)
        \le \varepsilon \|A\|_\infty \KL( \uX_i \uX_j^\tT, \uX_i \uX_j^\tT) 
        = 0,
    \end{align*}
    we conclude that $S^\varepsilon_{0_{n,n}} (\uX_i, \uX_j) = 0$. Thus, $\q G(0_{n,n}) = R_m$ and consequently
    \[
        \| \q G(B) \|_\infty
        \le 2 + 2 \varepsilon C + \| R_m \|_\infty
    \]
    so that the entries of $\q G(B) / \|\q G(B)\|_\infty$ satisfy
    \begin{align}
        \q G_{k,\ell}(B) / \|\q G(B)\|_\infty 
        \ge (R_m)_{k,\ell} / (2 + 2 \varepsilon C + \| R_m \|_\infty) > 0.
        \label{eq: boundbound G}
    \end{align}
    Since $\| \tilde{\q F}(A) \|_\infty = 1$, we obtain 
    \[
        \tilde{\q T}(\overline{\bb A}) = \tilde{\q G}( \tilde{\q F}(\overline{\bb A})) 
        \subseteq \{ A \in \bb A : A_{k,\ell} \ge (R_m)_{k,\ell} / (2 + 2 \varepsilon C + \| R_m \|_\infty) > 0, k \neq \ell, \ \| A \|_\infty =1 \} 
        \eqqcolon \bb K.
    \]
    Due to constraint $\| A \|_\infty =1$, the set
    $\bb K$ is compact, 
    but not convex. 
    For this reason, 
    we consider a convex and compact set
    \[
        \bb K_c 
        \coloneqq \{ A \in \bb A : A_{k,\ell} \ge (R_m)_{k,\ell} / (2 + 2 \varepsilon C + \| R_m \|_\infty) > 0, k \neq \ell, \ \| A \|_\infty \le 1 \} 
        \supset \bb K
    \]
    Then we have $\tilde{\q T}(\bb K_c) \subseteq \tilde{\q T}(\overline{\bb A}) \subseteq \bb K \subset\bb K_c$.
    Moreover, 
    by Lemma \ref{prop: sinkhorn properties} and Theorem \ref{l:lipshcitz}
    the operator $\tilde{\q T}$ is Lipschitz continuous and 
    consequently 
    continuous on $\bb K_c$.
    Then, 
    Brouwer's fixed point theorem, see Theorem \ref{thm: brouwer}, states 
    that there exists a fixed point $A \in \bb K_c$ of $\q T$. 
    In particular, 
    $\| A \|_\infty = \| \q T(A) \|_\infty = 1$ 
    and $A \in \bb K$.

    For the Wasserstein case, we have the same conclusions, just with Lipschitz constant $1$.
    \end{proof}

    \begin{remark} \label{rem:peyre_2}   
    Indeed, there are existence proofs for fixed points of $\tilde{\q T}$ for the Wasserstein case in the literature.
   Unfortunately, we are puzzled about the last step of the proofs 
    in \cite[Theorem 2.3]{huizing2022unsupervised}, \cite[Theorem 1]{lin2025coupled} 
    or  \cite[Lemma 2.2]{dusterwald2025fast}.
    In particular, 
    in \cite{huizing2022unsupervised,lin2025coupled}, 
    a nonconvex analog of the set $\bb K$ in our proof above is constructed, which is used for Brouwer's Fixed Point Theorem directly with an argument that it is (strongly) locally contractible. 
    The following small example shows that local contractibility is not sufficient:
        each point $x \in \mathbb S^1$ on the unit sphere corresponds 
        to some unique angle 
        $\theta_{x} \in [0, 2\pi)$ by $x = (\cos \theta_{x}, \sin \theta_{x})$. 
        The sphere is compact 
        and strongly locally contractible, 
        since for each neighborhood $\mathcal U$ of $x$ 
        we can find some $\delta > 0$,
        such that
        \begin{equation*}
            \{(\cos(\theta), \sin(\theta)) 
            \in \R^{2} : \theta \in\ (\theta_{x} - \delta, \theta_{x} + \delta) \} 
            \subset \mathcal U
        \end{equation*}
        is homeomorphic to an interval.
        If we now consider the rotation
        \begin{equation*}
            f: \mathbb S^1 \to \mathbb S ^1,\quad x \mapsto (\cos(\theta_{x} + t), \sin(\theta_{x} + t)), 
            \quad t \in \ (0, 2\pi),
        \end{equation*}
        we immediately see that $f$ cannot possess a fixed point by periodicity of cosine and sine.   

        In contrast to \cite{huizing2022unsupervised, lin2025coupled}, in \cite{dusterwald2025fast} different approach is taken. In the proof an analogy of $\bb K_c$ is constructed, 
        which includes $0_{m,m}$, 
        and on which Brouwer's fixpoint theorem is applied instead, similar to our proof. 
        While in our case the absence of $\| A\|_\infty$ in \eqref{eq: sinkhorn regularized} 
        yields continuity of $\q T$ on $\bb K_c$, 
        the proof in \cite{dusterwald2025fast} 
        uses definition \eqref{eq: wasserstein peyre} 
        for mappings $\q F$ and $\q G$. 
        Consequently,
        $\q F(0_{m,m}) = 0_{n,n}$ and $\tilde{\q F}(0_{m,m})$ is not well-defined 
        and $\q T$ cannot be continuous. However, 
        it is possible to obtain the existence of a fixed point 
        in the case of mapping \eqref{eq: wasserstein peyre}. 
        The proof relies on results from topology 
        and can be found in Appendix~\ref{sec: proof of existence}.
    \end{remark}

\section{Mahalanobis-Like Distances} \label{sec: kernel}
In this section, we use the same notation of the data as in the previous one,
but do not need the normalization \eqref{normalization}.
We are interested in positive definite and positive semidefinite matrices 
\[
        \pd^m
        \coloneqq 
        \{ A \in \bb R^{m \times m} : 
        A = A^T, 
        \; \xi^T A \xi  > 0 
        \;\text{for all}\;
        \xi \in \bb R^{m}, \ \xi \neq 0 \}
    \]   
and
    \[
        \psd^m \coloneqq 
        \{ A \in \bb R^{m \times m} : 
        A = A^T, 
        \; \xi^T A \xi  \geq 0 
        \;\text{for all}\;
        \xi \in \bb R^{m} \}.
    \]
We will use the notation $A \succeq 0$ for $A \in \psd^m$ and $A \succ 0$ for $A \in \pd^m$. On $\pd^m$ and $\psd^m$, the operation $\succeq $ defines a semiorder, so that $A \succeq B$ if $A - B \succeq 0$. To emphasize the dimensions, we use the abbreviations
$$
\mathbb A \coloneqq \pd^m \quad \text{and} \quad 
\mathbb B \coloneqq \pd^n.
$$
For $A \in \mathbb A$, the \emph{Mahalanobis distance} $M_A: \R^m \times \R ^m \to \R_{\ge 0}$ is defined by
    \begin{equation} \label{eq: mahalanobis}
	M_A(x, y) \coloneqq \sqrt{(x -y)^\tT A (x - y)}, \quad x,y \in \R^m.
	\end{equation}
For $A \in \overline{\mathbb A}$, the function $M_A$ is only a pseudo-distance.    
We are interested in functions
$\mathcal F: \R^{m \times m}  \to \R^{n \times n}$ 
and
$\mathcal G: \R^{n \times n}  \to \R^{m \times m}$ of the form
\begin{align}\label{eq:kernelregularized}
\mathcal F(A) \coloneqq f \big(M_A (\uX)\big) + R_n 
\quad \text{and} \quad 
\mathcal G(B) \coloneqq g \big(M_B (\overline{X}) \big) + R_m,
\end{align}
where $R_n \in \overline{\mathbb A}$ and $R_m \in  \overline{\mathbb B}$, and
    \begin{align} \label{eq: mahalanobis_matrix}
	M_A (\uX) \coloneqq \Big(  M_A(X_i,X_j)  \Big)_{i,j=1}^n
  \quad \text{and} \quad
    M_B (\overline{X}) \coloneqq \Big( M_B (X^k, X^{\ell}) \, \Big)_{k,\ell=1}^m,
	\end{align}
and the functions $f,g: \mathbb R \to \mathbb R_{>0}$ are applied componentwise. More precisely, we will restrict our attention to
\emph{radially positive definite functions} $f$ (and $g$)
satisfying
for all $m \in \N$, all $N \in \N$, all
$\{ x_k \}_{k=1}^N \subset \bb R^m$
    and  all $\{ \xi_k \}_{k=1}^N \subset \bb R$ 
the relation
    \begin{equation}\label{eq:kernel}
    \sum_{k,\ell =1 }^N \xi_k \xi_\ell f( \| x_k - x_\ell \|_2 ) \ge 0.
    \end{equation}
    In our numerical examples, we will choose $f=g$ as Gaussian or Laplacian functions, which are known to be radially positive definite. For  conditions on a function $f$ to be radially  positive definite, we refer to
    \cite{Schoenberg1946,Micchelli1986,W2004}.

Our mappings $\q F$ ad $\q G$ fulfill various properties summarized in the following lemmata.
    
    \begin{lemma}
    The mapping $\q F$ defined in \eqref{eq:kernelregularized} 
    maps $\overline{\bb A}$ to $\{ B \in \overline{\bb B} : B \succeq R_n \succeq 0 \}$. 
    Moreover, 
    if $R_n \in \bb B$, 
    then $\overline{\bb A}$ is mapped to $\{ B \in \bb B : B \succeq R_n \succ 0 \}$. 
    \end{lemma} 
    
    \begin{proof}
    Since $A \in \overline{\bb A}$, we have the
    Cholesky decomposition $A = L^T L$ with 
    $L \in \R^{m\times m}$. 
    Hence, we can rewrite $M_A(X_i,X_j)$ as
    \[
        M_A(X_i,X_j) 
        = \sqrt{ (X_i - X_j)^T L^T L (X_i - X_j)} 
        = \| L (X_i - X_j) \|_2. 
    \] 	
    Since $f$ is radially positive definite,
    applying \eqref{eq:kernel}  
    with $x_i = L X_i$ gives that $f(M_A) \succeq 0$.
    Moreover, 
    it holds $\q F(A) = f(M_A(\uX) ) + R_n \succeq R_n$.  
    \end{proof}
    
Concerning the continuity of $\q F$ we have the following results.
    
    \begin{lemma}\label{prop:kernelcont}
        Let $r_n \coloneqq \max_{i,j \in [n]} \| X_{i} - X_{j} \|_{1}^{2}$. Let $\q F$ and $\q G$ be defined by \eqref{eq:kernelregularized}.
        By $\lambda_n(R_n) \ge 0$ we denote 
        the smallest eigenvalue of $R_n$ 
        and set 
        \[
            q_n \coloneqq \lambda_n(R_n) \min_{i,j \in [n], i \neq j} \|X_i - X_j\|_2^2. 
        \]
        Then the following holds true:
    \begin{enumerate}
        \item If $f(\sqrt{\cdot}): \bb R_{\ge 0} \to \bb R$ 
        is $L$-Lipschitz, 
        then $\q F$ is $(r_n L)$-Lipschitz.
        \item If $f(\sqrt{\cdot}): [q_n, \infty) \to \bb R$ is $L$-Lipschitz, 
        then $\q F$ is $(r_n L)$-Lipschitz on $\q G(\overline{\bb B})$.
        \item If $f(\sqrt{\cdot}): [0, r_n] \to \bb R$ is $L$-Lipschitz, 
        then $\q F$ is $(r_n L)$-Lipschitz 
        on 
        
        $\{ A \in \overline{\bb A} :  \| A \|_\infty = 1 \}$.
        \item If $f(\sqrt{\cdot}): [q_n, r_n] \to \bb R$ is $L$-Lipschitz, 
        then $\q F$ is $(r_n L)$-Lipschitz on $\tilde{\q G} \left( \tilde{\q F}(\overline{\bb A}) \right)$ 
        with $\tilde{\q F},\tilde{\q G}$ in \eqref{fix2}.	  
    \end{enumerate}
    \end{lemma}
    
    \begin{proof}
        1. Using the Lipschitz continuity of $f(\sqrt{\cdot})$, 
        we get
        \begin{equation}
            |\q F_{i,j}(A) - \q F_{i,j}(A')| 
            = | f ( M_A(X_{i}, X_j) ) - f ( M_{A'}(X_i, X_j) ) |
            \le L |M_A^2(X_{i}, X_j) - M_{A'}^2(X_i, X_j)|.
        \end{equation}
        By definition of $M_A$, it holds 
        \begin{align*}
            |M_A^2(X_{i}, X_j) - M_{A'}^2(X_i, X_j)|
            & = \left| (X_{i} - X_{j})^{T} A (X_{i} - X_{j}) - (X_{i} - X_{j})^{T} A' (X_{i} - X_{j}) \right| \\
            &= \left| (X_{i} - X_{j})^{T} (A - A') (X_{i} - X_{j}) \right|\\
            & =  \left| \langle A - A', (X_{i} - X_{j}) (X_{i} - X_{j})^{T} \rangle_{F} \right| \\
            & \le \| (X_{i} - X_{j})(X_{i} - X_{j})^{T} \|_{1} \|A - A' \|_\infty \\
            &= \| X_{i} - X_{j} \|_{1}^{2} \|A - A'\|_\infty
        \end{align*}
    so that
    \[
        \| \q F(A) - \q F(A')\|_\infty 
        \le L \max_{i,j \in [n]} \| X_{i} - X_{j} \|_{1}^{2} \|A - A'\|_\infty.
    \]
    
    2. Since $M_A(X_{i}, X_i)=0$ for all $i \in [n]$, we get 
        \[
            |\q F_{i,i}(A) - \q F_{i,i}(A')| 
            = |f(0) - f(0)| 
            = 0 
            \le L_f \max_{i,j \in [n]} \| X_{i} - X_{j} \|_{1}^{2} \|A - A'\|_\infty.
        \]
    When $i,j \in [n]$, $i \neq j$, for $A \in \q G(\overline{\bb B})$, we have $A \succeq R_m$. This yields
    \[
        M_A^2(X_{i}, X_j) = (X_{i} - X_{j})^{T} A (X_{i} - X_{j})
        \ge (X_{i} - X_{j})^{T} R_m (X_{i} - X_{j})
        \ge \lambda_n(R_m) \|X_i - X_j\|_2^2 \ge q_n.
    \]
   Now we can argue as in part 1 with $f(\sqrt{\cdot})$ being Lipschitz continuous in only on $[q_n, \infty)$.
    
    3. By assumption $\| A \|_\infty = 1$ and we obtain 
    \[
        M_A^2(X_{i}, X_j) 
        =  \langle A, (X_{i} - X_{j}) (X_{i} - X_{j})^{T} \rangle_{F}
        \le \| X_{i} - X_{j} \|_{1}^{2} \le r_n.
    \]
    Consequently,    the assertion follows as in part 1.
    
    4. Combining parts 2) and 3) yields the assertion.
    \end{proof}
    
    For normalized iterations \eqref{fix2}, 
    we also need by Proposition \ref{l:lipshcitz}
    that $\|\q F\|_\infty$ is bounded away from zero. 
    By the following lemma, this can be achieved by selecting the reference $R_n$ appropriately.
    
    \begin{lemma}\label{prop: kernel lower bound}
     Let $\q F$ be defined by \eqref{eq:kernelregularized}.   Then it holds
        \[
            \| \q F(A) \|_\infty 
            \ge f(0) + \max_{i \in [n]} (R_n)_{i,i} \ge 0. 
        \]
    \end{lemma} 
    
    \begin{proof}
        Taking $N=1$, $\xi_1 = 1$ and arbitrary $x_1$ in \eqref{eq:kernel} gives
        $f(0) \ge 0$. Furthermore, we have $R_{i,i} \ge 0$, since $R$ is positive semidefinite. Then we get
        \begin{align*}
        \| \q F(A) \|_\infty & = \max_{i,j \in [n]} |\q F_{i,j}(A)| 
        \ge \max_{i \in [n]} |\q F_{i,i}(A)| 
        = \max_{i \in [n]} |f(M_A(X_i,X_i)) + R_{i,i}| \\
        & = \max_{i \in [n]} |f(0) + R_{i,i}| 
        = f(0) + \max_{i \in [n]} R_{i,i} \ge 0.
        \end{align*}
    \end{proof}
    
    To put the above lemmas into perspective, we consider the following examples.
    
    \begin{example}     \label{ex: rbf}
    1.  Let $f(t) \coloneqq e^{-\frac{t^2}{2 \sigma^2}}$, $\sigma^2 > 0$ be the Gaussian function.
        Then, for $0 \le t \le s$, we get
        \[
            | f(\sqrt{t}) - f(\sqrt{s}) | 
            = e^{- \frac{t}{2 \sigma^2}} - e^{- \frac{s}{2 \sigma^2}}
            =  e^{- \frac{t}{2 \sigma^2}} (1 - e^{\frac{t-s}{2 \sigma^2}})
            \le e^{- \frac{t}{2 \sigma^2}} \left( 1 - (1+\tfrac{t-s}{2 \sigma^2}) \right) 
            \le \tfrac{s-t}{2 \sigma^2},
        \] 
        where the inequality $e^{a} \ge 1 + a$ 
        for all $a \in \bb R$ was used. 
        For $t > s \ge 0$, we interchange the role of $t$ and $s$ and obtain that $f(\sqrt{\cdot})$ is $\tfrac{1}{2\sigma^2}$-Lipschitz continuous. 
        Thus, by Lemma \ref{prop:kernelcont}, the function
        $\q F$ is Lipschitz continuous with constant $L = r_n / 2 \sigma^2$. 
        Taking $R_n \coloneqq \tau_\q F I_n$, $\tau_\q F \ge 0$,
        Lemma~\ref{prop: kernel lower bound} provides
        \begin{equation} \label{ex1}
            \|\q F(A) \|_\infty
            \ge f(0) + \tau_\q F 
            = 1 + \tau_\q F.
        \end{equation}
    2.  Let $f(t) \coloneqq (1 + (\varepsilon t)^2)^{-1/2}$, $\varepsilon > 0$
        which is radially positive definite, see the inverse multi-quadratic kernel \cite[p. 54]{schoelkopf2001kernellearning}.
        Since $\tau \mapsto \sqrt{1 + \varepsilon^2\tau}$
        is $\frac{\varepsilon^2}{2}$-Lipschitz on $\R_{\ge 0}$, for $0 \leq s,t$, we have 
        \begin{align*}
           |f(\sqrt{t}) - f(\sqrt{s})|
            = \left|\tfrac{\sqrt{1 + \varepsilon^2 s} - \sqrt{1 + \varepsilon^2 t}}{\sqrt{(1 + \varepsilon^2 t)(1 + \varepsilon^2 s)}}\right|
            < \varepsilon^2/2 |t - s|.
        \end{align*}
    For $R_n \coloneqq \tau_\q F I_n$, $\tau_\q F \ge 0$, we have again \eqref{ex1}.    
    \\
    3.  Let $f(t) \coloneqq e^{-\frac{|t|}{\sigma}}$, $\sigma > 0$ 
        be the Laplacian function.
        For $0 \le t \le s$, 
        there exists $t < \theta < s$ such that by the mean value theorem 
        \[
            | f(\sqrt{t}) - f(\sqrt{s}) | 
            \le \left| \tfrac{1}{2 \sigma \sqrt{\theta} } e^{- \frac{\sqrt{\theta}}{\sigma}}  (t- s) \right| 
            \le \tfrac{1}{2 \sigma \sqrt{\theta} } |t - s|.
        \] 
        As $\theta^{-1/2}$ can get arbitrarily large in the proximity of 0, 
        $f(\sqrt{\cdot})$ is not Lipschitz continuous on $\bb R_{\ge 0}$. 
        However, 
        if $R_n = \tau_\q F I_n$ with $\tau_\q F > 0$, 
        the constant $q_n$ from Lemma~\ref{prop:kernelcont} is given by $q_n = \tau_\q F \min_{i,j \in [n]} \|X_i - X_j\|_2^2 > 0$ by $X_i \neq X_j$. 
        Then, 
        it suffices to consider $t,s \ge q_n$ 
        and we obtain by monotonicity of $\theta^{-1/2}$ that
        \[
            | f(\sqrt{t}) - f(\sqrt{s}) | 
            \le \tfrac{1}{2 \sigma \sqrt{q_n} } |t - s|,
        \]
        so that
        $\q F$ is $L$-Lipschitz with $L = r_n / 2 \sigma \sqrt{q_n}$ on $[q_n, \infty)$.
        We have the same lower bound as in \eqref{ex1}.        \end{example}
        
Now we summarize convergence results for the algorithms in Section \ref{sec: fix point}. The proof follows directly 
from the Theorems \ref{l:lipshcitz}, \ref{thm: convergence grad}
and \ref{thm: convergence stoch}.
    
    \begin{theorem}  \label{cor: kernel-mahalanobis}
    Let  $\q F$ and $\q G$ be defined by \eqref{eq:kernelregularized}
    with $R_m \in \bb A$ and $R_n \in \bb B$. 
    Assume that $\q F$ and $\q G$ are Lipschitz continuous with constants  $L_\q F$ and $L_\q G$. 
        Let $A^0 \in \overline{\bb A}$, $B^0 \in \overline{\bb B}$. 
        Then the following holds true: 
        \begin{enumerate}
            \item if $\max_{i 
    \in [n]} (R_n)_{i,i} > 2 L_\q F - f(0)$ and $\max_{k 
    \in [m]} (R_m)_{k,k} > 2 L_\q G - g(0)$, then
            the sequence $(A^t,B^t)_t$ generated by Algorithm \ref{alg:fix2}
            converges to the unique fixed point $(A,B) \in \bb A \times \bb B$ of \eqref{fix2}.
            \item if $0 < \gamma_\q F, \gamma_\q G$ and $\gamma_\q F \gamma_\q G < 1/L_\q F L_\q G$, then 
            the sequence $(A^t,B^t)_t$ generated by Algorithm \ref{alg:fix}
            converges to the unique fixed point
            $(A, B) \in \bb A \times \bb B$ of \eqref{fix1}.
            \item if $0 < \gamma_\q F \le \sqrt \alpha/ \sqrt 2 m L_\q F$ and
            $0< \gamma_\q G \le 1/n L_\q G$, $m,n >1$, then
            the sequence $(A^t,B^t)_t$ generated by Algorithm~\ref{alg: rfi} 
            converges almost surely to a unique fixed point $(A, B) \in \bb A \times \bb B$ 
            of \eqref{fix1}.
        \end{enumerate}
    \end{theorem}

    \section{Graph Laplacian Distances} \label{sec: laplacian}
In this section, we deal with functions $\q F$ and $\q G$ which are just linear in $A$ and $B$.
To this end, let
\begin{align} \label{eq: laplacian}
	\mathcal W_A (\uX) \coloneqq \Big(  M_A^2(X_i,X_j)  \Big)_{i,j=1}^n
  \quad \text{and} \quad
    \mathcal W_B (\overline{X}) \coloneqq \Big( M_B^2 (X^k, X^{\ell}) \, \Big)_{k,\ell=1}^m.
	\end{align}
Further, let $\diag(w)$ denote the diagonal matrix
with $w$ on its main diagonal.
We are interested in functions
$\mathcal F: \R^{m \times m}  \to \R^{n \times n}$ 
and
$\mathcal G: \R^{n \times n}  \to \R^{m \times m}$ of the form
\begin{align}\label{eq:graphl}
\mathcal F(A) &\coloneqq \diag\big(\mathcal W_A (\uX) \1_n \big) - \mathcal W_A (\uX),
\\
\mathcal G(B) &\coloneqq \diag\big(\mathcal W_B (\overline{X}) \1_m \big) - \mathcal W_B (\overline{X}).
\end{align}

\begin{remark}	\label{rem:graph}
The above functions are known as graph Laplacian:
given $A \in\mathbb A$,
we consider a weighted graph $\mathfrak G = \mathfrak G(A) = (\q V,\ \q E,\ \q W)$ 
with vertices 
$\q V = [n]$, $\q E \subset [n]^2$ 
and weights 
$\q W_{i,j} = (X_i - X_j)^\tT A (X_i - X_j) = M_A^2(X_i,X_j)$. 
The edge $(i,j)$ is present in $\mathfrak G(A)$ 
if and only if the corresponding weight $\q W_{i,j}$ is positive. 
The graph Laplacian matrix $\q L$ 
associated with $\mathfrak G$ 
is given by
\[
    \q L
    \coloneqq 
    \diag(\q W \1_n) - \q W , \quad \mathcal W \coloneqq (\mathcal W_{i,j})_{i,j=1}^n \in \bb R^{n \times n},
\] 
It is well known that $\q L$ is a symmetric positive semidefinite matrix 
with the smallest eigenvalue $0$ 
and corresponding eigenvector $\1_n$. 
The second smallest eigenvalue of $\q L$ 
is positive if and only if $\mathfrak G$ is connected. 
Further properties of graph Laplacians can be found, e.g., in \cite{chung1997spectral}. There is a close connection with transfer operators, e.g., from optimal transport plans 
\cite{OTCoherentSet2021}.
\end{remark}

By the remark, for 
$$
\mathbb A \coloneqq \pd^m \quad \text{and} \quad 
\mathbb B \coloneqq \pd^n,
$$
the following result is straightforward.

\begin{lemma}\label{cor: laplacian image}
The mapping $\q F$ defined in \eqref{eq:graphl} 
maps $\overline{\bb A}$ to $\overline{\bb B}$.
\end{lemma}

In the previous Sections \ref{sec: wasserstein} 
and \ref{sec: kernel} 
the mappings $\q F, \q G$ were nonlinear 
and the machinery derived in Section \ref{sec: fix point} 
was required to derive the existence of fixed points. 
In contrast, the mappings $\q F, \q G$ in \eqref{eq:graphl}, and thus also their concatenation, are linear. Therefore, this eigenvalue/eigenvector can be found by standard methods from linear 
algebra. In our numerical examples, we will just use the power iteration method.

In the next theorem, we will apply the 
Perron-Frobenius theorem to a special part
of $\mathcal G \circ \q F $ to show that this operator has a simple largest eigenvalue 
$\lambda_{\mathcal G \circ \mathcal F} >0$ and that there exists indeed a graph Laplacian $A$ with
$$
\lambda_{\mathcal G \circ \mathcal F} A = \mathcal G\left( \mathcal F (A) \right).
$$
Then, choosing $\lambda_\q F,\lambda_\q G > 0$ with
$\lambda_\q F\, \lambda_\q G = \lambda_{\mathcal G \circ \mathcal F}$ and
setting 
$$
 \lambda_\q F B \coloneqq 
  \q F(A)
\quad \text{and} \quad
\lambda_\q G A \coloneqq \q G(B)
$$
we get a solution of \eqref{fix1}, where
$\gamma_{\q F} = \frac{1}{\lambda_\q F}$ and
$\gamma_{\q G} = \frac{1}{\lambda_\q G}$.


\begin{theorem}\label{thm: existence graph Laplace}
Assume for the entries of $X \in \R^{m\times n}$ that for all $k,\ell ,p,q \in [m]$, $k \neq \ell$, $p \neq q$, there exist $i,j \in [n]$ such that
\begin{equation} \label{abc}
\left(  X_{i,k} - X_{i,\ell} - X_{j,k} + X_{j,\ell} \right) \left( X_{i,p} - X_{i,q} - X_{j,p} + X_{j,q} \right) \neq 0.
\end{equation}
Let $\q F$ and $\q G$ be defined by \eqref{eq:graphl}.
Then, that there exists an eigenvector $A \in \overline{\bb A}$  of the linear operator $\q G \circ \q F$ 
corresponding to its largest eigenvalue $\lambda_{\q G \circ \q F} >0$. 
\end{theorem}

\begin{proof}
    Given that $A$ is symmetric 
    and its diagonal entries satisfying 
    $A_{k,k} = - \sum_{\ell = 1, \ell \neq k}^{m} A_{k,\ell}$, 
    we transform $\q G(\q F(A))$ 
    into matrix applied to the entries of $A$ 
    lying below the main diagonal. 
    To make the notation condensed,
    we will use the pairwise differences
    \[
        \Delta^{k,\ell} \coloneqq X^k - X^\ell \in \bb R^n
        \quad \text{and} \quad
        \Delta_{i,j} \coloneqq X_i - X_j \in \bb R^m.
    \]
 Then, with Kronekers delta function $\delta_{k,\ell}$, we can rewrite the entries of $\q G(B)$ 
    for arbitrary $B \in \bb B$ as
    \begin{align*}
        \q G_{k,\ell}(B) 
        & = \delta_{k,\ell} (\q W_B(\overline X) \1_m)_{k} - (\q W_{B})_{k,\ell}(\overline X)
        = \delta_{k,\ell} \sum_{t = 1}^m (\q W_{B}(\overline X))_{k,\ell} - (\q W_{B}(\overline X))_{k,\ell} \\
        & = \delta_{k,\ell} \sum_{t = 1}^m M_B^2(X^k,X^t) - M_B^2(X^k,X^\ell) 
        = \delta_{k,\ell} \sum_{t = 1}^m (\Delta^{k,t})^{\tT} B \Delta^{k,t}  - (\Delta^{k,\ell})^{\tT} B \Delta^{k,\ell} \\
        & = \delta_{k,\ell} \sum_{t = 1}^m \sum_{i,j=1}^n  \Delta^{k,t}_i B_{i,j} \Delta^{k,t}_j  - \sum_{i,j=1}^n \Delta^{k,\ell}_i B_{i,j} \Delta^{k,\ell}_j\\
        & = \sum_{i,j=1}^n \left[ \delta_{k,\ell} \sum_{t = 1}^m  \Delta^{k,t}_i \Delta^{k,t}_j - \Delta^{k,\ell}_i \Delta^{k,\ell}_j \right] B_{i,j}
    \end{align*}
    for all $k,\ell \in [m]$.
    Thus, we obtain 
    \[
    \q G_{k,\ell}(\q F(A)) = 
    \sum_{i,j=1}^n \left[ \delta_{k,\ell} \sum_{t = 1}^m  \Delta^{k,t}_i \Delta^{k,t}_j - \Delta^{k,\ell}_i \Delta^{k,\ell}_j \right] \q F_{i,j}(A).
    \]
    Substituting 
    \begin{align*}
        \q F_{i,j}(A)
        & = 
        \delta_{i,j} \sum_{s = 1}^n \Delta_{i,s}^{\tT} A \Delta_{i,s}  - \Delta_{i,j}^{\tT} A \Delta_{i,j} \\
        & = 
        \delta_{i,j} \sum_{s = 1}^n \langle A,  \Delta_{i,s} \Delta_{i,s}^{\tT} \rangle_F  - \langle A, \Delta_{i,j} \Delta_{i,j}^{\tT} \rangle_F = 
        \bigg\langle A,  \delta_{i,j} \sum_{s = 1}^n  \Delta_{i,s} \Delta_{i,s}^{\tT} - \Delta_{i,j} \Delta_{i,j}^{\tT} \bigg\rangle_F,
    \end{align*}
    this becomes further
    \begin{align*}
        \q G_{k,\ell}(\q F(A)) 
        & = 
        \bigg\langle A, \underbracket{\sum_{i,j=1}^n \left[ \delta_{k,\ell} \sum_{t = 1}^m  \Delta^{k,t}_i \Delta^{k,t}_j - \Delta^{k,\ell}_i \Delta^{k,\ell}_j \right] \cdot 
        \left[ \delta_{i,j} \sum_{s = 1}^n  \Delta_{i,s} \Delta_{i,s}^{\tT} - \Delta_{i,j} \Delta_{i,j}^{\tT} \right]}_{\eqqcolon H^{k,\ell}} \bigg\rangle_F,
    \end{align*}
    so that 
    $\q G_{k,\ell}(\q F(A)) = \langle A, H^{k,\ell} \rangle_F$.
    Note that by $\Delta^{k,\ell}_i = - \Delta^{\ell,k}_i$,
    we have 
    $H^{k,\ell} = H^{\ell,k}$ 
    and $H^{k,k} = - \sum_{\ell= 1, \ell \neq k} H^{k,\ell}$ 
    reflecting the dependencies between entries of $A$. 
    Moreover, 
    $H^{k,\ell}$ is symmetric 
    as a sum of weighted symmetric matrices $\Delta_{i,j} \Delta_{i,j}^{\tT}$.
    Hence,  from now on we only consider the lower triangular part of $A$ 
    with indices $k \in [m]$, $\ell \in [k-1]$
    and rewrite the inner product as
    \begin{align*}
        \langle A, H^{k,\ell} \rangle_F
        & = \sum_{\substack{p,q =1 \\ p \neq q}}^m A_{p,q} H^{k,\ell}_{p,q} + \sum_{p =1}^m A_{p,p} H^{k,\ell}_{p,p} 
        = \sum_{\substack{p,q =1 \\ p \neq q}}^m A_{p,q}  H^{k,\ell}_{p,q} 
        - \sum_{\substack{p,q =1 \\ p \neq q}}^m A_{p,q} H^{k,\ell}_{p,p} \\
        & = \sum_{\substack{p,q =1 \\ p \neq q}}^m A_{p,q}  H^{k,\ell}_{p,q} 
        - \sum_{\substack{p,q =1 \\ p \neq q}}^m A_{p,q} \left( \frac{1}{2} H^{k,\ell}_{p,p} + \frac{1}{2} H^{k,\ell}_{q,q} \right) \\
        & = - \sum_{\substack{p,q =1 \\ p \neq q}}^m A_{p,q} \left( \frac{1}{2} H^{k,\ell}_{p,p} - H^{k,\ell}_{p,q} + \frac{1}{2} H^{k,\ell}_{q,q} \right)
        = - \sum_{\substack{p,q =1 \\ p > q}}^m A_{p,q} \left(  H^{k,\ell}_{p,p} - 2 H^{k,\ell}_{p,q} + H^{k,\ell}_{q,q} \right).
    \end{align*}
    Let us now consider double indexed vectors 
    \[
        a 
        \coloneqq 
        \{ A_{p,q} \}_{p =2, q = 1}^{m,p -1} 
        \quad \text{and} \quad 
        h_{(k,\ell)} 
        \coloneqq 
        \left( - H^{k,\ell}_{p,p} + 2 H^{k,\ell}_{p,q} - H^{k,\ell}_{q,q} \right)_{p =2, q = 1}^{m,p-1}
    \]
    in $\bb R^{\tfrac{m(m-1)}{2}}$.
    The matrix $\zb H \in \bb R^{\tfrac{m(m-1)}{2} \times \tfrac{m(m-1)}{2}}$,
    collects all row-vectors $h_{(k,\ell)}^{\tT}$ 
    for $k \in [m], \ell \in [k-1]$. 
    Then,
    our eigenvector problem reads as $\zb H v = \lambda v$. 
    
    To argue that a solution $a$ exists, 
    we will use the Perron-Frobenius, see  Theorem~\ref{thm:fp} in the appendix.
    For this, 
    we show that entries of $\zb H$ are positive. 
    Since $\ell < k$, we have
    for $p \in [m]$ and $q \in [p-1]$ that
    \[
        H^{k,\ell}_{p,q} 
        = - \sum_{i,j=1}^n \Delta^{k,\ell}_i \Delta^{k,\ell}_j  \cdot 
        \left[ \delta_{i,j} \sum_{s = 1}^n  (\Delta_{i,s})_p (\Delta_{i,s})_q - (\Delta_{i,j})_p (\Delta_{i,j})_q \right],
    \]
    and hence
    \begin{align*}
       \zb H_{(k,\ell), (p,q)} &=   (h_{(k,\ell)})_{p,q}
        = - H^{k,\ell}_{p,p} + 2 H^{k,\ell}_{p,q} - H^{k,\ell}_{q,q} \\
        & = \sum_{i,j=1}^n \Delta^{k,\ell}_i \Delta^{k,\ell}_j 
        \big[ \delta_{i,j} \sum_{s = 1}^n  \left(  (\Delta_{i,s})_p - (\Delta_{i,s})_q \right)^2 - \left( (\Delta_{i,j})_p - (\Delta_{i,j})_q \right)^2 \big].
    \end{align*}
    Separating the cases $i = j$ and $i \neq j$ gives
    \begin{align*}
       \zb H_{(k,\ell), (p,q)} 
        & = \sum_{i=1}^n (\Delta^{k,\ell}_i )^2 \sum_{s = 1}^n  \left( (\Delta_{i,s})_p - (\Delta_{i,s})_q \right)^2
        - \sum_{\substack{i,j=1 \\ i \neq j} }^n \Delta^{k,\ell}_i \Delta^{k,\ell}_j \left( (\Delta_{i,j})_p - (\Delta_{i,j})_q \right)^2. 
    \end{align*}
    Since for $s = i$ the summand is zero, otherwise
    the change in order of summation in the first term gives
    \[
        \sum_{\substack{i,s=1 \\ i \neq s}}^n (\Delta^{k,\ell}_i )^2  \left( (\Delta_{i,s})_p - (\Delta_{i,s})_q \right)^2
        = \sum_{\substack{i,s=1 \\ i \neq s}}^n \left( \frac{1}{2} (\Delta^{k,\ell}_i )^2 + \frac{1}{2} (\Delta^{k,\ell}_s )^2 \right) \left( (\Delta_{i,s})_p - (\Delta_{i,s})_q \right)^2.
    \]
    Renaming $s$ into $j$ leads to
    \begin{align*}
        \zb H_{(k,\ell), (p,q)} 
        & = \sum_{\substack{i,j=1 \\ i \neq j} }^n  \left( \frac{1}{2} (\Delta^{k,\ell}_i )^2 - \Delta^{k,\ell}_i \Delta^{k,\ell}_j  + \frac{1}{2} (\Delta^{k,\ell}_j )^2 \right) \left( (\Delta_{i,j})_p - (\Delta_{i,j})_q \right)^2 \\
        & = \frac{1}{2} \sum_{\substack{i,j=1 \\ i \neq j} }^n (  \Delta^{k,\ell}_i - \Delta^{k,\ell}_j )^2 \left( (\Delta_{i,j})_p - (\Delta_{i,j})_q \right)^2 \\
        & = \frac{1}{2} \sum_{i,j=1}^n  ( \Delta^{k,\ell}_i - \Delta^{k,\ell}_j )^2 \left( (\Delta_{i,j})_p - (\Delta_{i,j})_q \right)^2
        > 0,
    \end{align*}
    where  at least one summand is nonzero by assumption on $X$, 
    Now the Perron-Frobenius theorem 
    states the existence of an eigenvalue $\lambda > 0$, 
    which is the unique largest eigenvalue of $\zb H$ 
    and the corresponding eigenvector $v \in \bb R^{m(m-1)/2}$ satisfies $v_{k,\ell} > 0$. 
    \\
    Since by construction $A$ is a graph Laplacian matrix, 
    its off-diagonal entries have to be nonpositive.
    Thus, we set
        \[
        A_{k,\ell} \coloneqq - v_{k,\ell}, \ A_{\ell, k} \coloneqq A_{k,\ell}
        \quad \text{and} \quad
        A_{k,k} \coloneqq - \sum_{\substack{t =1\\ t \neq k}} A_{k,t},
        \quad k \in [m], \ \ell \in [k-1]. 
    \]  
   Since $A$ is a graph Laplacian matrix, we know that 
    $A \in \overline{\mathbb A}$.
\end{proof}

Notably, it follows from the theorem that $M_A$ is a metric matrix, which may be of interest on its own.

\begin{corollary}\label{cor: Laplacian metric}
Let $X \in \R^{m \times n}$ fulfill \eqref{abc} and in addition 
$X_i - X_j \notin \spn\{ \1_m \}$ for all $i,j \in [n],\ i \neq j$. 
Let $\q F$ and $\q G$ be defined by \eqref{eq:graphl}, and let
$\lambda_{\q F \circ \q G} A = \q G(\q F(A)$, where
$\lambda_{\q F \circ \q G} >0$ is the largest eigenvalue of $\q G \circ \q F$.
Then the matrix  $M_A(\uX)$ is a distance matrix, 
i.e., $M_A(\uX) \in \bb D_n$.
\end{corollary}

\begin{proof}
The matrix $M_A(\uX)$, $A \in \overline{\mathbb A}$, is only guaranteed to be a pseudo-metric.
    Setting $B \coloneqq \q F(A)$, 
    we have $\lambda_{\q G \circ \q F} A = \q G(B)$ and further by \eqref{eq:graphl} it holds
    $\q W_{B, k,\ell}(\overline X) = -  A_{k,\ell} > 0$
    for $k,\ell \in [m]$, $k \neq \ell$.
    Then Remark \ref{rem:graph} gives that the underlying graph $\mathfrak G(B)$ is connected and $A$ has zero as eigenvalue of multiplicity $1$ with corresponding eigenvector $\1_m$. Thus, $0 = M_A^2(X_i, X_j) = (X_i - X_j)^T A (X_i - X_j)$ if and only if $X_i - X_j \in \spn(\1_m)$. By assumption, this case is excluded and thus for all $i,j \in [n]$, $i \neq j$, $M_A(X_i, X_j) > 0$, i.e., $M_A$ is a metric matrix. 	
\end{proof}

Finally, we give a remark on the assumptions 
on the data $X$.

\begin{remark}
The assumptions \eqref{abc} in Theorem \ref{thm: existence graph Laplace} and $X_i - X_j \notin \spn\{ \1_m \}$ from Corollary \ref{cor: Laplacian metric} are satisfied if, for instance, the entries of $X$ are sampled independently
from an absolutely continuous probability measure.  
More generally, the assumptions are satisfied for all generic $X$,
meaning that there exists a polynomial $P$ such than $P(X) = 0$ if and only if the assumptions are not satisfied. To see this, we set $P \coloneqq P_1 + P_2$, where both are nonnegative polynomials describing the first and the second assumption,
\begin{align*}
    P_1(X) 
    & \coloneqq 
    \prod_{\substack{k,\ell = 1\\ k \neq \ell}}^m \prod_{\substack{p,q = 1\\ p \neq q}}^m \left[ \sum_{i,j =1}^n \left(  X_{i,k} - X_{i,\ell} - X_{j,k} + X_{j,\ell} \right)^2 \left( X_{i,p} - X_{i,q} - X_{j,p} + X_{j,q} \right)^2 \right], \\
    P_2(X) 
    & \coloneqq 
    \prod_{\substack{i,j=1 \\ i \neq j}}^n \left[ \sum_{k,\ell=1}^m (X_{i,k} - X_{j,k} - X_{i,\ell} + X_{j,\ell})^2 \right]. 
\end{align*}
\end{remark}

\section{Numerical Results} \label{sec:numerics}

In this section, we numerically explore eigenvalue methods discussed in the previous sections. First, we study their performance on synthetic data of translated histograms, see Section \ref{sec: synthtic_data}. 
Secondly, we explore their performance on the single-cell RNA sequencing (scRNA-seq) dataset from \cite{stegle2015sccellseq} in Section \ref{sec: cell_data}.
The code for all experiments is available\footnote{\url{https://github.com/JJEWBresch/unsupervised_ground_metric_learning}}.

\subsection{Translated Histograms}    \label{sec: synthtic_data}

\paragraph{Dataset.}

In the first experiment, we validate the proposed methods on three synthetic datasets $\{X^{(k)}\}_{k = 1}^3 \subseteq \mathbb R^{n\times m}$ similarly to \cite{huizing2022unsupervised}.
Here $n \in \mathbb N$ is the number of samples and $m \in \mathbb N$ is the number of features.
The datasets $X^{(k)}$ are generated by circulant-like sampling 
$$
    X^{(k)}_{i,j} \coloneqq h_k(\tfrac{i}{n} - \tfrac{j}{m}),
    \quad 
    (i,j) \in [n] \times [m], \ k \in \{1,2,3\},
$$ 
of  periodic functions $h_k$ on the one-dimensional torus $[0,1]$ given by
$$
    h_1(x) \;\propto\; e^{-\tfrac{10^4 |x|^2}{25}},
    \quad
    h_2(x) \;\propto\; h_1(x) + \tfrac12 h_1(x+\tfrac12),
    \quad 
    h_3(x) \;\propto\; h_1(x) + \tfrac12 h_1(x+\tfrac13).
$$
The parameters are chosen as $n = 100$ and $m = 80$. Since the samples $X^{(k)}_{i}$ are obtained by shifting of $X^{(k)}_{1}$, we expect the resulting learned distances to be small whenever the peaks of $h_1$ are aligned and large otherwise.

\paragraph{Algorithms.}
For these datasets, we compute the following eigenvectors: 
\begin{itemize}
    \item Wasserstein eigenvectors (WEV) corresponding to the mappings \eqref{eq:kernelregularized_w} with $R_n = \tau \underline R$, and $R_m = \tau \overline R$, where $\tau = 10^{-2}$, $\underline R_{i,j} = \| \uX_i^{(k)} - \uX^{(k)}_j\|_1$ and $\overline R_{p,\ell} = \| (\oX^{(k)})^{p} - (\oX^{(k)})^{\ell}\|_1$ with $\uX$ and $\oX$ defined in \eqref{normalization} with $\| \underline{R} \|_\infty = 1.9665$ and $\| \overline{R} \|_\infty = 2.4575$.
     
    \item Sinkhorn eigenvectors (SEV) corresponding to the mappings \eqref{eq: sinkhorn regularized} with $\varepsilon = 5 \cdot 10^{-2}$ and the rest of the parameters are the same as for WEV;
    \item Mahalanobis eigenvectors with radial basis function kernel (RBF-MEV) corresponding to Example \ref{ex: rbf} with transforms $\q F$ and $\q G$ using the same kernel parameters $\sigma_\q F = \sigma_\q G = 1$ and $\tau_\q F = \tau_\q G = 0.01$;
    \item Graph Laplacian Mahalanobis eigenvectors (GMEV) as in \eqref{eq:graphl};
    \item Euclidean distance (Eucl) as a baseline metric. 
\end{itemize}

We distinguish two variants for each of the algorithms. The first is a non-normalized with mapping $\q T$ as in \eqref{eq: operator Q} with $\alpha = 0.9$ and $\gamma_\q F = \gamma_\q G = 0.75$ for the WEV and SEV and the RBF-MSEV and GMEV with $\alpha = 0.9$ and $\gamma_\q F = \gamma_\q G = 0.01$. The second is a normalized iteration corresponding to the mapping $\tilde{\q T}$ in \eqref{fix2}, in which case we add suffix "n" to the name, e.g., WEVn.

For this experiment, all computations are performed on an off-the-shelf MacBookPro 2020 with Intel Core i5 Chip (4‑Core CPU, 1.4~GHz) and 8~GB~RAM.

\begin{figure}[t!]
    \includegraphics[width = 1\textwidth]{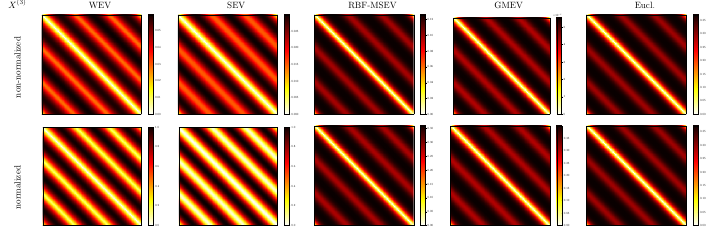}
    \caption{Comparison of the considered eigenvalue methods. Here, the resulting distances $\dist(X_i, X_j)$ are depicted for the dataset $X^{(3)}$.}
    \label{fig: metrics_synthetic_data_norm_heat}
\end{figure}

\begin{figure}[b!]
    \includegraphics[width = 1\textwidth]{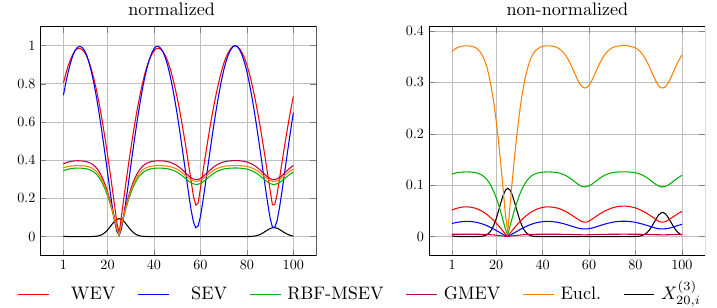}
    \caption{
    Distance $\dist(X^{(3)}_{20}, X^{(3)}_i)$ for the normalized (left) and non-normalized (right) eigenvector computation schemes. 
    \label{fig: metrics_synthetic_data_norm_comp}}
\end{figure}

The resulting learned distances $\dist(X_i^{(3)}, X_j^{(3)})$ are depicted in Figure \ref{fig: metrics_synthetic_data_norm_heat} and slice of the distance for 20th sample, i.e., $\dist(X_{20}^{(3)}, X^{(3)}_i)$, is shown in Figure \ref{fig: metrics_synthetic_data_norm_comp}. We observe that learned distances exhibit a circulant-like structure similar to that of the dataset. Besides the small values close to the main diagonal, we observe two off-diagonals corresponding to shifts by $\pm \tfrac{1}{3}$ and alignment of the peaks in $h_3$. This is more pronounced for optimal transport-based methods and, in particular, for WEVn and SEVn, while learned Mahalanobis distances (RBF-MEV, GMEV) are similar to the Euclidean distance.

The established theoretical results are validated by Figure \ref{fig: conv_synthetic_data_norm} depicting the $\ell_\infty$-residual $\|A^t - A^{t+1}\|_\infty$ and Hilbert metric on the manifold $\R^{m \times m}_{>0}$. Since, in our case, diagonals are always zero, we exclude the diagonals, yielding
\begin{equation}\label{eq: Hilbert metric}
    d_\q H(A^t, A^{t+1}) 
    \coloneqq 
    \max_{\substack{k,\ell \in [m] \\ k \neq \ell}} \log(A^t_{k,\ell} / A^{t+1}_{k,\ell}) 
    - \min_{\substack{k,\ell \in [m] \\ k \neq \ell}} \log(A^t_{k,\ell} / A^{t+1}_{k,\ell}).
\end{equation}

\begin{figure}[t!]
    \includegraphics[width = 1\textwidth]{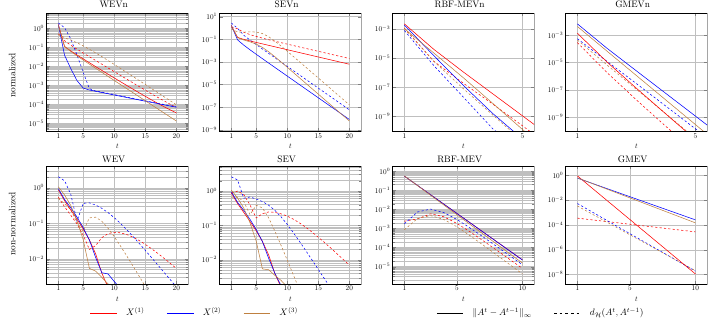}
    \caption{$\ell_\infty$-residual and Hilbert norm progression with iterations for the metric learning methods from Fig.~\ref{fig: conv_synthetic_data_norm} for the different synthetic data sets $\{X^{(i)}\}_{i = 1}^3$.}
    \label{fig: conv_synthetic_data_norm} 
\end{figure}

The computation time for the WEV (in both cases) is within 4 min., whereas the computation for the SEV (in both cases) is within 1 min, highlighting the computational benefit of entropic regularization. Its smoothing effect is also observed in Figure \ref{fig: metrics_synthetic_data_norm_comp}, where the learned metric for SEV has less sharp transitions than the metric obtained by WEV.
Mahalanobis distances (RBF-MEV, GMEV) are computed within 25 sec.

\subsection{Clustering Single-Cell Dataset} \label{sec: cell_data}

\paragraph{Dataset.}

For the next trials, we use a preprocessed scRNA-seq dataset \cite{wolf2018cell} containing 2043 cells and 1030 genes. For a detailed description of preprocessing, we refer to \cite[Appx. H]{huizing2022unsupervised}. Each cell belongs to exactly one of the following classes depending on its biological function: ``B cell'' (B), ``Natural Killer'' (NK), ``CD4+ T cell'' (CD4 T), ``CD8+ T cell'' (CD8 T), ``Dendritic cell'' (DC) and ``Monocyte'' (Mono). In addition, canonical marker genes expressed in certain cell types are annotated according to Azimuth\footnote{\url{https://azimuth.hubmapconsortium.org/references/}} \cite{hao2021integrated}.

We also work with the second version of the dataset obtained by the PCA reduction with 10 principal components for genes. After PCA, entries may be negative, and we additionally apply the exponential transform $X_{i,k} \mapsto \exp(X_{i,k})$. As a result, we obtain $X \in \R^{2043 \times 10}_{>0}$ that we refer to as the reduced dataset. Note that using PCA reduction for clustering is a standard technique \cite{zha2001spectral}, and we can see learning a metric for the reduced dataset as its enhancement.


\paragraph{Algorithms.}

We continue to work with the eigenvector algorithms for the synthetic dataset, except for WEV due to its computational complexity. The parameters in SEV and SEVn are chosen as $\varepsilon = 10^{-1}$, $R_n = \tau \underline R$ and $R_m = \tau \overline R$ with $\tau = 10^{-3}$. The norm are $\| \underline{R} \|_\infty = 1.9912$ and $\| \overline{R} \|_\infty = 1.9612$ for reduced dataset and $\| \underline{R} \|_\infty = 1.9658$ and $\| \overline{R} \|_\infty = 1.9959$ for full dataset. In addition, we include two versions of stochastic SEV. The first is sSEVn that was proposed in \cite{huizing22git} and sSEV corresponding to Algorithm \ref{alg: rfi}. For a stochastic iteration, 10\% of entries are selected for update randomly without replacement. For non-stochastic and stochastic methods, we perform at most 15 and, respectively, 400 iterations. 

For the non-normalized approaches, we chose $\alpha = 0.9$, $\gamma = \gamma_{\q F} = \gamma_{\q G}$ and study two setting: theoretically justified $\gamma = 0.9$ and not covered $\gamma= 1.0$. 

Additionally, we include the RBF-MEVn with $\sigma_{\q F} = 10$ and $\sigma_{\q G} = 1$ and regularization parameter $\tau = 10^{-3}$. Furthermore, we also test normalized Laplacian kernel Mahalanobis eigenvector (Laplacian-MEVn) from Example \ref{ex: rbf} with the same parameters as RBF-MEVn. GMEVn approach and Euclidean distance are included for the comparison. 

All experiments were performed on a cluster with Xeon E5-2630v4 (CPU) and NVIDIA Tesla P100 (GPU).

\paragraph{Performance metrics.}

To evaluate the quality of learned metrics, we use them for clustering. With known labels for cells and genes, the following performance metrics are computed.

\textit{Average silhouette width (ASW).} For a sample $x_i$ from class $C_p$, the mean distance to the points in the same class $a_i$ and the mean distance to the closest class $b_i$ are given by
\[
a_i = \frac{1}{|C_p|-1} \sum_{j \in C_p} \dist(x_i, x_j)
\quad \text{and} \quad
b_i = \min_{q \neq p}\frac{1}{|C_q|} \sum_{j \in C_q} \dist(x_i, x_j).
\]
Then, ASW is obtained by
\[
\mathrm{ASW} = \frac{1}{n} \sum_{i \in [n]} \frac{b_i - a_i}{\max\{a_i,b_i\}}.
\]
It takes values in $[-1, 1]$ with larger ASW indicating better cluster separation. 

\textit{Dunn index (DI).} 
It is the ratio of the minimal distance between the points in different classes to the maximal distance within the classes. The exact formula is given by
\[
\textrm{DI} = \frac{d_{min}}{d_{max}},
\quad \text{with} \quad 
d_{min} = \min_{p \neq q} \min_{i \in C_p, \, j \in C_q} \dist(x_i, x_j)
\ \text{and} \
d_{max} = \max_{p} \max_{i,j \in C_p} \dist(x_i,x_j),
\]
so that larger values indicate better clustering.

\textit{t-Distributed Stochastic Neighbor Embedding
 (t-SNE) \cite{maaten2008tsne}.} 
For the visualization of learned clusters via the computed distance matrices, we use t-SNE projections of the dataset in a two-dimensional space with random initialization.
For reduced dataset, we use higher regularization to prevent curve-like embeddings.

\begin{table}[b!]
    \resizebox{\linewidth}{!}{
    \begin{tabular}{l @{\qquad}  l l l l @{\qquad} l l l l @{\qquad}  l l}
        \toprule
        \multirow{2}{*}{Method} & \multicolumn{4}{l}{\!\!\!Reduced (cell)} & \multicolumn{4}{l}{\!\!\!Full (cell)} & \multicolumn{2}{l}{\!\!\!Full (gene)} \\
        & ASW$\uparrow$ & DI$\uparrow$ & time, m & iter. & ASW$\uparrow$ & DI$\uparrow$ & time, m & iter. 
        & ASW$\uparrow$ & DI$\uparrow$ \\
        \midrule
        SEVn
        & 0.7521 & 0.0006 & 57 & 15 
        & 0.3507 & 0.0487 & 269 & 15
        & 0.1659 & 0.0024 \\
        sSEVn \cite{huizing2022unsupervised}
        & 0.7328 & 0.0006 & 299 & 400 
        & \textbf{0.3894} & 0.0525 & 1440 & 400
        & \textbf{0.2281} & 0.0030 \\
        \midrule
        SEV with $\gamma = 0.9$
        & 0.7200 & $\mathbf{0.0079}$ & 37 & 15 
        & 0.0666 & 0.4574 & 175 & 15
        & 0.1881 & \textbf{0.1880} \\
        sSEV with $\gamma = 0.9$
        & 0.7142 & 0.0024 & 11 & 54 
        & 0.0989 & 0.0000 & 711 & 400 
        & 0.0523 & 0.0003 \\
        SEV with $\gamma = 1.0$
        & $\mathbf{0.7698}$ & 0.0044 & 81 & 15 
        & 0.0666 & 0.4342 & 345 & 15 
        & 0.1762 & 0.1244 \\
        sSEV with $\gamma = 1.0$
        & 0.7696 & 0.0009 & 164 & 400 
        & 0.0667 & 0.0000 & 967 & 400 
        & 0.0441 & 0.0001 \\
        SEV with $\tau = 10^{-5}$
        & --- & --- & --- & --- 
        & 0.3552 & 0.0574 & 12 & 4  
        & 0.1881 & 0.0684 \\
        SEV with adapt. $\gamma_{\q F} = 1.11, \gamma_{\q G} = 13.71$
        & --- & --- & --- & --- 
        & 0.3541 & 0.0436 & 44 & 15  
        & 0.1613 & 0.0023 \\
        \midrule
        RBF-MEVn
        & 0.2291 & 0.0001 & 3 & 3 
        & 0.0863 & 0.1277 & 202 & 3
        & -0.0095 & 0.0644 \\
        Laplacian-MEVn
        & 0.2231 & 0.0001 & 3 & 3
        & 0.0802 & \textbf{0.5443} & 128 & 3
        & 0.0033 & 0.1393 \\
        GMEVn
        & 0.2207 & 0.0007 & 5 & 3 
        & 0.2102 & 0.2187 & 62 & 10
        & 0.0205 & 0.1033\\
        \midrule
        Eucl.
        & 0.2154 & 0.0015 & 1 & ---
        & 0.0732 & 0.4333 & 4 & ---
        & 0.0033 & 0.0790\\
        \bottomrule
    \end{tabular}}
    \caption{Comparison of clustering of cells and genes with learned metrics on the PCA-reduced and the full scRNAseq datasets.}
    \label{tab: red_cell}
\end{table}

\begin{figure}[t!]
    \includegraphics[width = 1\textwidth]{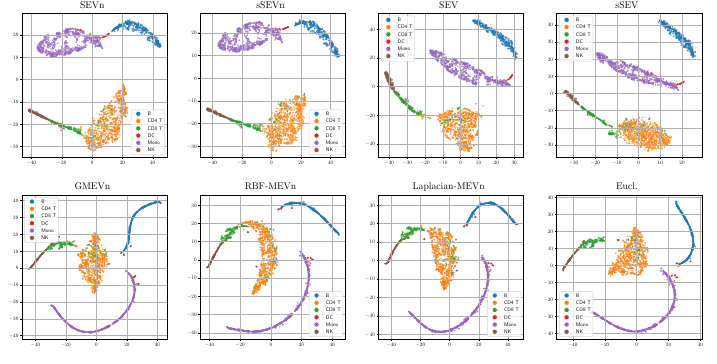}
    \caption{t-SNE plots of the resulting cell clustering 
    using learned metrics on the PCA-reduced scRNAseq dataset.}
    \label{fig: red_cell_cluster}
\end{figure}

\begin{figure}[b!]
    \includegraphics[width = 1\textwidth]{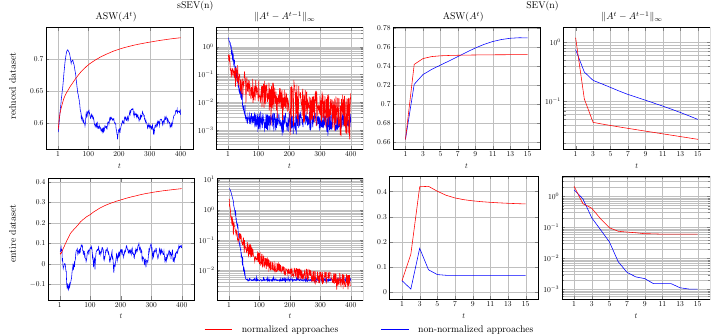}
    \caption{Evolution of ASW and the $\ell_\infty$-residual with iteration for optimal transport methods for PCA-reduced (top) and full (bottom) scRNAseq datasets.}
    \label{fig: red_cell_conv}
\end{figure}

\paragraph{Comparison of proposed methods}

The summary of clustering performance for the different methods is presented in Table~\ref {tab: red_cell}. We observe that the performance of RBF-MEVn and Laplacian-MEVn is close to the Euclidean distance, performing marginally better in ASW and requiring longer runtime. 
We also performed (unreported) grid search for the kernel parameters $\sigma_\q F$ and $\sigma_\q G$, which had little impact on the clustering statistics.

The performance of GMEVn is similar to kernel methods on the reduced dataset. We also observe this on t-SNE plots shown in Figure \ref{fig: red_cell_cluster}.
For the full dataset, GMEVn exhibits the best ASW for cell clustering among Mahalanobis learned distances.

\begin{figure}[t!] 
   \includegraphics[width = 1\textwidth]{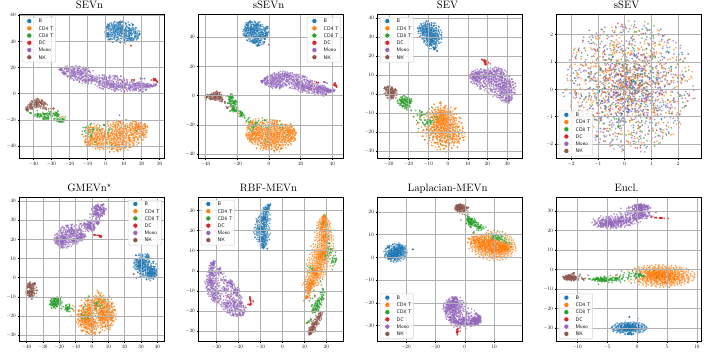}
   \caption{t-SNE visualization of cell clusters with learned distances on the full scRNAseq dataset.}
   \label{fig:tsne_vis_entire_dataset}
\end{figure}

\begin{figure}[b!] 
   \includegraphics[width = 1\textwidth]{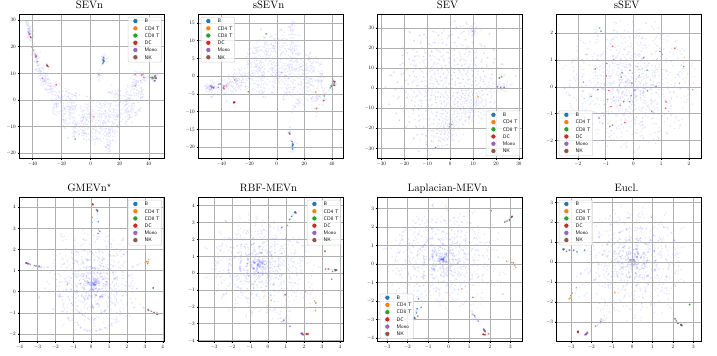}
   \caption{t-SNE visualization of the marker gene clusters using learned distances on the full scRNAseq dataset.}
   \label{fig:tsne_vis_entire_dataset_gene}
\end{figure}

Turning to OT-based distances, we again observe a dichotomy between the performance on the reduced and full datasets. On the reduced dataset, all Sinkhorn-based methods attain ASW close to 0.75. The t-SNE plots, shown in Figure \ref{fig: red_cell_cluster}, indicate a slight difference between normalized and non-normalized methods. In these plots, cell clusters are clearly visible, with difficulties in separating CD8 T cells from CD4 T and NK cells and DC cells from Mono and B cells. Time-wise, SEV with $\gamma= 0.9$ and its stochastic version are the fastest. Note that we stopped sSEV after 54 iterations at the peak ASW shown in Figure \ref{fig: red_cell_conv}.

For the full dataset, normalized methods still exhibit high ASW for both cell and gene clustering. On the other hand, non-normalized approaches no longer result in high ASW for cell clustering. On the contrary, they lead to higher DI for both cell and gene clustering. In contrast to sSEVn, sSEV does not provide informative clustering as can be seen in Figures \ref{fig:tsne_vis_entire_dataset} and \ref{fig:tsne_vis_entire_dataset_gene}. The rest of the methods cluster cells in the full dataset visually similar to the reduced dataset. There is still difficulty separating CD8 T cells from CD4 T and NK cells, however, DC cells are now forming a distinct cluster. Since we also observe that marker genes cluster into groups in Figure \ref{fig:tsne_vis_entire_dataset_gene}, yet these groups cannot be separated from the rest of the unlabeled genes. 

\paragraph{Adaptive step size for SEV}

\begin{figure}[b!]
\begin{minipage}[t]{0.53\linewidth}
    \includegraphics[width=\linewidth, clip=true, trim=0pt 0pt 0pt 0pt]{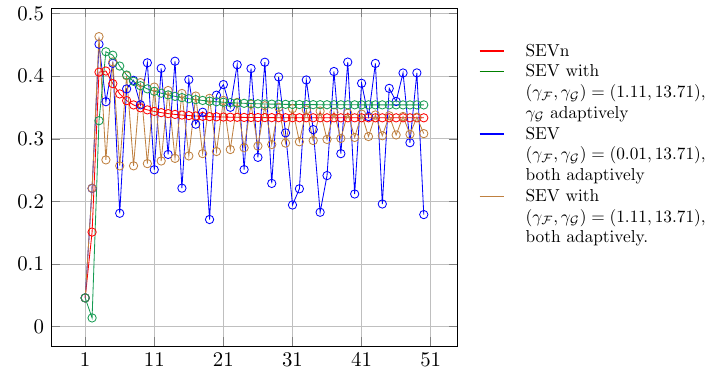}
    \caption{
    ASW development for SEVn and SEV
    for different choices of $\gamma_\q F$ and $\gamma_\q G$
    and adaptive updates according to the $\ell_\infty$-residual,
    where $\tau = 10^{-3}$ is constant.
    }
    \label{fig: adaptive_gammas}
\end{minipage}%
\hfill
\begin{minipage}[t]{0.45\linewidth}
    \includegraphics[width=\linewidth, clip=true, trim=0pt 0pt 0pt 0pt]{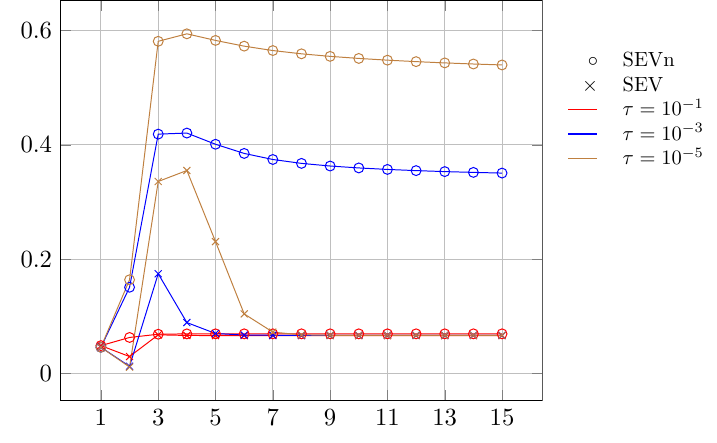}
    \caption{
    ASW development for SEVn and SEV
    for different choices of $\tau > 0$,
    and $\gamma = 0.9$ is constant.
    }
    \label{fig: ASW_sinkhorn_early_stop}
    \vfill
\end{minipage}
\end{figure}

Investigating convergence in Figure \ref{fig: red_cell_conv}, we observe that despite a linear convergence rate in residual, ASW for sSEV varies around 0.05, while for sSEVn it slowly increases over time. The runtime, however, is much longer than for non-stochastic SEVn, showing analogous performance. Although SEVn performance is the best among the methods, we observe that it does not converge, and the residual $\|A^t - A^{t+1}\|_\infty$ stagnates at $10^{-1}$. It is linked to the fact that $\|R_n\|_\infty, \|R_m\|_\infty \approx 2 \tau = 2 \cdot 10^{-3}$ and conditions in Theorem \ref{cor: wasserstein} do not apply. This indicates that the better performance of SEVn compared to SEV could be due to theoretical restrictions on the choice of $\gamma$ for the latter. 

We explore this idea by considering a version of SEV with $\gamma_\q F = \| \q F(A) \|_\infty^{-1} = 1.11$ and $\gamma_\q G = \| \q G(B) \|_\infty^{-1} = 13.71$, where $(A,B)$ is the outcome of SEVn. With these $\gamma_\q F, \gamma_\q G$, SEV diverge and for this reason, we included an adaptive change of $\gamma_\q F$ and $\gamma_\q G$ based on the evolution of $\|B^{t} - \q F (A^t)\|_\infty$ and, respectively, of $\|A^{t} - \q G (B^t)\|_\infty$. More precisely, after each iteration, we rescale $\gamma_{\q F}$ by $\|B^{t} - \q F (A^t)\|_\infty^{-1}$ and $\gamma_{\q G}$ by $\|A^{t} - \q G (B^t)\|_\infty^{-1}$ and use them for the next iteration of Algorithm \ref{alg:fix}.
Figure \ref{fig: adaptive_gammas} shows ASW progression for the adaptive strategies, which matches ASW of SEVn, indicating that performance in this case depends on the initial choice of $\gamma_\q F$ and $\gamma_\q G$. The t-SNE plots for SEV with adaptive step size can be seen in Figure \ref{fig: tSNE_sinkhorn_early_stop}. 
The runtime of SEV with adaptive strategy, reported in Table \ref{tab: red_cell}, is 6 times faster than SEVn.

\paragraph{Impact of $\tau$ and early stopping of SEV}

\begin{figure}[t!]
    \includegraphics[width=\linewidth, clip=true, trim=0pt 0pt 0pt 0pt]{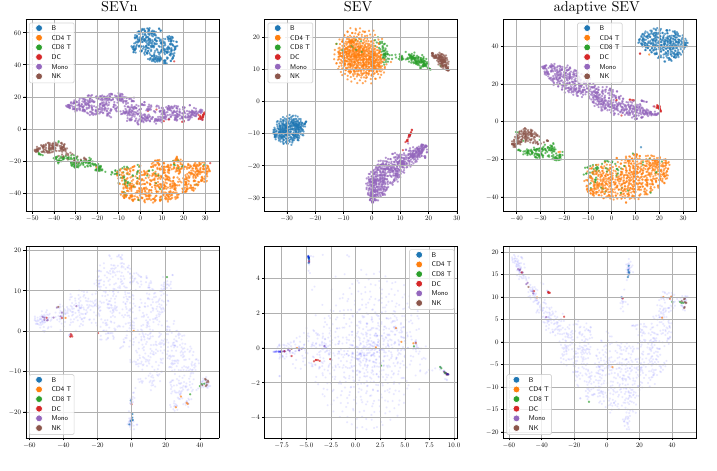}
   \caption{t-SNE visualization of the resulting clustering of  
   cells (top) and marker genes (bottom) using learned distances on the full scRNAseq dataset. Left: SEVn with $\tau = 10^{-5}$, center: early stopped SEV with $\tau = 10^{-5}$, right: SEV with adaptive step size and $\tau = 10^{-3}$.}
   \label{fig: tSNE_sinkhorn_early_stop}
\end{figure}

Due to constraints of Theorem \ref{cor: wasserstein} not being satisfied for SEVn, we explored the performance of SEV and SEVn for different values of $\tau$. The change of ASW with iterations is shown in Figure \ref{fig: ASW_sinkhorn_early_stop}. Both methods attain higher ASW for $\tau = 10^{-5}$ 
and for SEV, we observe that ASW initially increases and then vanishes. A similar, but less drastic behavior is also visible for SEVn. 

These observations motivated us to consider SEV with $\tau = 10^{-5}$ stopped at the peak ASW. 
Since the residual is still large for early stopped SEV, this procedure can be seen as a regularization only giving $A \approx \gamma_\q G\q G(B)$ and $B \approx \gamma_\q F\q F(A)$, i.e., allowing for error in order to counteract noise in the data. This method is also included in Table \ref{tab: red_cell} and respective t-SNE plots for SEVn and early stopped SEV with $\tau = 10^{-5}$ can be found in Figure \ref{fig: tSNE_sinkhorn_early_stop}.

\section{Conclusions}\label{sec:conclusions}
We considered two fixed-point problems
\begin{equation} \label{fix2c}
    B =  \frac{ \q F(A)}{\|\q F(A)\|_\infty } 
    \quad \text{and} \quad 
    A =  \frac{\q G(B)}{\|\q G(B)\|_\infty }. 
\end{equation}
and
\begin{equation} \label{fix1_c}
    B = \gamma_\q F \q F(A) 
    \quad \text{and} \quad 
    A = \gamma_\q G  \q G(B).
\end{equation}
For the first problem, we used Algorithm \ref{alg:fix2} and, for the second, we established Algorithm \ref{alg:fix} and the stochastic Algorithm \ref{alg: rfi}. We established general linear convergence guarantees of the proposed algorithms and more detailed results for 
\begin{enumerate}
\item (regularized) optimal-transport-based mappings $\q F$ and $\q G$ in Section \ref{sec: wasserstein};
\item kernel-based Mahalanobis distances in Section \ref{sec: kernel};
\item Graph Laplacian-based Mahalanobis distance in Section \ref{sec: laplacian}.
\end{enumerate}
Most of our results should hold for any vector norm $\|\cdot\|$ and the specific choice of the $\ell_\infty$-norm was motivated by the entrywise definition of $\q F$ and $\q G$ combined with the Lipschitz continuity of the respective distances discussed in Sections \ref{sec: wasserstein} and \ref{sec: kernel}.   

Our numerical trials indicate that optimal transport-based distances have strong clustering potential, especially combined with the PCA reduction of the data. Graph Laplacian-based Mahalanobis distance also shows promise for larger datasets. On the other hand, kernel-based distances improve upon Euclidean distance only marginally.  

In the numerical trials, we observed that the equalities in \eqref{fix2c} and \eqref{fix1_c} as hard constraints may cause algorithms to underperform. In the future, we plan to explore alternative relaxed constraints, e.g., via unbalanced optimal transport \cite{chizat2018unbalanced}. Furthermore, there is a definite space for improvement of the stochastic Algorithm \ref{alg: rfi}.

We also observed that Algorithm \ref{alg:fix2} converges in scenarios beyond those covered by theoretical results. It remains an open problem to justify its performance. We believe that this requires a different distance from $\| \cdot \|_{\infty}$, possibly induced by a manifold of chosen parametrization.

\subsection*{Acknowledgments}
The authors are grateful to Geert-Jan Huitzing for providing the preprocessed dataset for numerical trials.


\printbibliography

@book{athreya06measure,
author = {Krishna B. Athreya, Soumendra N. Lahiri},
title = {Measure Theory and Probability Theory},
publisher = {Springer New York, NY},
year = {2006},
doi = {10.1007/978-0-387-35434-7},
address = {New York, NY},
edition   = {1},
}

@article{hermer2019random,
  title={Random function iterations for consistent stochastic feasibility},
  author={Hermer, Neal and Luke, D Russell and Sturm, Anja},
  journal={Numerical Functional Analysis and Optimization},
  volume={40},
  number={4},
  pages={386--420},
  year={2019},
  publisher={Taylor \& Francis},
  doi={10.1080/01630563.2018.1535507},
}

@unpublished{hermer2020random,
  title={Random function iterations for stochastic fixed point problems},
  author={Hermer, Neal and Luke, D. Russell and Sturm, Anja},
  year={2020},
  doi={10.48550/arXiv.2007.06479},
}

@Article{OTCoherentSet2021,
  author  = {Koltai, Peter and von Lindheim, Johannes and Neumayer, Sebastian and Steidl, Gabriele},
  journal = {Physica D},
  title   = {Transfer operators from optimal transport plans for coherent set detection},
  year    = {2021},
  pages   = {132980},
  volume  = {426},
  doi     = {10.1016/j.physd.2021.132980},
}

@book{W2004,
  title={Scattered Data Approximation},
  author={Holger Wendland},
  series={Cambridge Monographs on Applied and Computational Mathematics},
  publisher={Cambridge University Press},
  year={2005},
  hideisbn={9780511617539},
  doi={10.1017/CBO9780511617539},
}

@article{Schoenberg1946,
author = {Schoenberg, I.J},
title = {Contributions to the problem of approximation of equidistant data by analytic functions},
journal = {Quarterly of Applied Mathematics},
sjournal = {Quart. Appl. Math.},
volume = {4},
year = {1946},
pages = {45-99}
}

@article{Micchelli1986,
author     = {C. A. Micchelli}, 
title      = {Interpolation of scattered data: distance matrices and conditionally
positive definite functions},
journal    = {Constructive Approximation},
volume = {2},
pages = {11--22}, 
year = {1986},
doi={10.1007/BF01893414},
}

@article{stuart2020inverse,
  title   = {Inverse Optimal Transport},
  author  = {Stuart, Andrew M. and Wolfram, Marie‐Therese},
  journal = {SIAM Journal on Applied Mathematics},
  volume  = {80},
  number  = {1},
  pages   = {257--279},
  year    = {2020},
  doi     = {10.1137/19M1261122}
}

@InProceedings{huizing2022unsupervised,
  title = 	 {Unsupervised Ground Metric Learning Using {W}asserstein Singular Vectors},
  author =       {Huizing, Geert-Jan and Cantini, Laura and Peyr{\'e}, Gabriel},
  booktitle = 	 {Proceedings of the 39th International Conference on Machine Learning},
  pages = 	 {9429--9443},
  year = 	 {2022},
  editor = 	 {Chaudhuri, Kamalika and Jegelka, Stefanie and Song, Le and Szepesvari, Csaba and Niu, Gang and Sabato, Sivan},
  volume = 	 {162},
  series = 	 {Proceedings of Machine Learning Research},
  month = 	 {17--23 Jul},
  publisher =    {PMLR},
  url = 	 {https://proceedings.mlr.press/v162/huizing22a.html},
}

@article{makrodimitris2019metric,
    author = {Makrodimitris, Stavros and Reinders, Marcel J T and van Ham, Roeland C H J},
    title = {Metric learning on expression data for gene function prediction},
    journal = {Bioinformatics},
    volume = {36},
    number = {4},
    pages = {1182-1190},
    year = {2019},
    month = {09},
    issn = {1367-4803},
    doi = {10.1093/bioinformatics/btz731},
    hideurl = {https://doi.org/10.1093/bioinformatics/btz731},
    hideeprint = {https://academic.oup.com/bioinformatics/article-pdf/36/4/1182/48982814/bioinformatics\_36\_4\_1182\_s5.pdf},
}

@misc{huizing22git,
  author = {Huizing, Geert-Jan},
  title = {Python package \texttt{wsingular}},
  year = {2022},
  publisher = {GitHub},
  journal = {GitHub repository},
  howpublished = {\url{https://github.com/CSDUlm/wsingular}},
  note = {version:0.1.7}
}

@book{horn2012matrix,
  author    = {Roger A. Horn and Charles R. Johnson},
  title     = {Matrix Analysis},
  edition   = {2nd},
  publisher = {Cambridge University Press},
  year      = {2012},
  isbn      = {9780521386326},
}

@book{chung1997spectral,
  title     = {Spectral Graph Theory},
  author    = {Fan R. K. Chung},
  publisher = {American Mathematical Society},
  year      = {1997},
  series    = {CBMS Regional Conference Series in Mathematics},
  volume    = {92},
  isbn      = {978-0-8218-0315-8},
  doi       = {10.1090/cbms/092},
}

@inproceedings{dusterwald2025fast,
title={Fast Unsupervised Ground Metric Learning with Tree-Wasserstein Distance},
author={Kira Michaela D{\"u}sterwald and Samo Hromadka and Makoto Yamada},
booktitle={The Thirteenth International Conference on Learning Representations},
year={2025},
url={https://openreview.net/forum?id=FBhKUXK7od}
}

@article{cuturi2014ground,
  author  = {Marco Cuturi and David Avis},
  title   = {Ground Metric Learning},
  journal = {Journal of Machine Learning Research},
  year    = {2014},
  pages   = {533--564},
  url     = {http://jmlr.org/papers/v15/cuturi14a.html},
  hidedoi     = {10.5555/2627435.2627452},
  publisher = {JMLR}, 
  volume = {15}, 
  number = {1},
  numpages = {32},
}

@article{wolf2018cell,
author = {Wolf, F. and Angerer, Philipp and Theis, Fabian},
year = {2018},
month = {02},
pages = {},
title = {SCANPY: Large-scale single-cell gene expression data analysis},
volume = {19},
journal = {Genome Biology},
doi = {10.1186/s13059-017-1382-0}
}

@InProceedings{chen2020simple,
  title = 	 {A Simple Framework for Contrastive Learning of Visual Representations},
  author =       {Chen, Ting and Kornblith, Simon and Norouzi, Mohammad and Hinton, Geoffrey},
  booktitle = 	 {Proceedings of the 37th International Conference on Machine Learning},
  pages = 	 {1597--1607},
  year = 	 {2020},
  editor = 	 {III, Hal Daumé and Singh, Aarti},
  volume = 	 {119},
  month = 	 {13--18 Jul},
  hidepdf = 	 {http://proceedings.mlr.press/v119/chen20j/chen20j.pdf},
  url = 	 {https://proceedings.mlr.press/v119/chen20j.html},
  hidedoi = {10.5555/3524938.3525087},
  publisher = {JMLR}, 
  articleno = {149}, 
  numpages = {11}, 
  series = {ICML'20}
}

@inproceedings{khodadadeh2019unsupervised,
 author = {Khodadadeh, Siavash and Boloni, Ladislau and Shah, Mubarak},
 booktitle = {Advances in Neural Information Processing Systems},
 editor = {H. Wallach and H. Larochelle and A. Beygelzimer and F. d\textquotesingle Alch\'{e}-Buc and E. Fox and R. Garnett},
 pages = {},
 publisher = {Curran Associates, Inc.},
 title = {Unsupervised Meta-Learning for Few-Shot Image Classification},
 url = {https://proceedings.neurips.cc/paper_files/paper/2019/file/fd0a5a5e367a0955d81278062ef37429-Paper.pdf},
 volume = {32},
 year = {2019}
}

@article{kaya2019deep,
  title={Deep Metric Learning: A Survey},
  author={Kaya, Mahmut and Bilge, Hasan {\c{S}}akir},
  journal={Symmetry},
  volume={11},
  number={9},
  pages={1066},
  year={2019},
  publisher={MDPI},
  DOI = {10.3390/sym11091066}
}

@unpublished{ghojogh2022spectral,
  title={Spectral, probabilistic, and deep metric learning: Tutorial and survey},
  author={Ghojogh, Benyamin and Ghodsi, Ali and Karray, Fakhri and Crowley, Mark},
  doi={10.48550/arXiv.2201.09267},
  year={2022}
}

@misc{vu2021deep,
title = {Deep Metric Learning: a (Long) Survey},
author = {Chan Ha Vu},
year = {2021},
URL = {https://hav4ik.github.io/articles/deep-metric-learning-survey}
}

@article{stegle2015sccellseq,
  title={Computational and analytical challenges in single-cell transcriptomics},
  author={Oliver Stegle and Sarah A. Teichmann and John C. Marioni},
  journal = {Nature Reviews Genetics},
  volume = {16},
  number = {3},
  pages = {133--145},
  year = {2015},
  doi = {10.1038/nrg3833},
}

@INPROCEEDINGS{hadsell2006dimensionality,
  author={Hadsell, R. and Chopra, S. and LeCun, Y.},
  booktitle={2006 IEEE Computer Society Conference on Computer Vision and Pattern Recognition (CVPR'06)}, 
  title={Dimensionality Reduction by Learning an Invariant Mapping}, 
  year={2006},
  volume={2},
  number={},
  pages={1735-1742},
  doi={10.1109/CVPR.2006.100}
}

@InProceedings{schroff2015facenet,
author = {Schroff, Florian and Kalenichenko, Dmitry and Philbin, James},
title = {FaceNet: A Unified Embedding for Face Recognition and Clustering},
booktitle = {Proceedings of the IEEE Conference on Computer Vision and Pattern Recognition (CVPR)},
month = {June},
year = {2015},
pages     = {815--823},
doi       = {10.1109/CVPR.2015.7298682},
}

@INPROCEEDINGS{liu2017sphereface,
  author={Liu, Weiyang and Wen, Yandong and Yu, Zhiding and Li, Ming and Raj, Bhiksha and Song, Le},
  booktitle={2017 IEEE Conference on Computer Vision and Pattern Recognition (CVPR)}, 
  title={SphereFace: Deep Hypersphere Embedding for Face Recognition}, 
  year={2017},
  volume={},
  number={},
  pages={6738-6746},
  doi={10.1109/CVPR.2017.713}
}

@INPROCEEDINGS{wang2018cosface,
  author={Wang, Hao and Wang, Yitong and Zhou, Zheng and Ji, Xing and Gong, Dihong and Zhou, Jingchao and Li, Zhifeng and Liu, Wei},
  booktitle={2018 IEEE/CVF Conference on Computer Vision and Pattern Recognition}, 
  title={CosFace: Large Margin Cosine Loss for Deep Face Recognition}, 
  year={2018},
  volume={},
  number={},
  pages={5265-5274},
  doi={10.1109/CVPR.2018.00552}
}

@InProceedings{kusner2015from,
  title = 	 {From Word Embeddings To Document Distances},
  author = 	 {Kusner, Matt and Sun, Yu and Kolkin, Nicholas and Weinberger, Kilian},
  booktitle = 	 {Proceedings of the 32nd International Conference on Machine Learning},
  pages = 	 {957--966},
  year = 	 {2015},
  editor = 	 {Bach, Francis and Blei, David},
  volume = 	 {37},
  series = 	 {Proceedings of Machine Learning Research},
  address = 	 {Lille, France},
  month = 	 {07--09 Jul},
  publisher =    {PMLR},
  hidepdf = 	 {http://proceedings.mlr.press/v37/kusnerb15.pdf},
  url = 	 {https://proceedings.mlr.press/v37/kusnerb15.html},
  hidedoi = {10.5555/3045118.3045221}
}

@unpublished{bellazzi2021gene,
  title={The Gene Mover's Distance: Single-cell similarity via Optimal Transport},
  author={Bellazzi, Riccardo and Codegoni, Andrea and Gualandi, Stefano and Nicora, Giovanna and Vercesi, Eleonora},
  doi={10.48550/arXiv.2102.01218},
  year={2021}
}

@inproceedings{xing2002distance,
 author = {Xing, Eric and Jordan, Michael and Russell, Stuart J and Ng, Andrew},
 booktitle = {Advances in Neural Information Processing Systems},
 editor = {S. Becker and S. Thrun and K. Obermayer},
  pages={505--512},
 title = {Distance Metric Learning with Application to Clustering with Side-Information},
 volume = {15},
 year = {2002},
url={https://papers.nips.cc/paper/2164-distance-metric-learning-with-application-to-clustering-with-side-information}
}

@book{ghojogh2023elements,
  title     = {Elements of Dimensionality Reduction and Manifold Learning},
  author    = {Ghojogh, Benyamin and Crowley, Mark and Karray, Fakhri and Ghodsi, Ali},
  publisher = {Springer},
  address = {New York, NY},
  year      = {2023},
  hideisbn      = {978-3-031-10602-6},
  doi       = {10.1007/978-3-031-10602-6},
}

@inproceedings{davis2007information,
author = {Davis, Jason V. and Kulis, Brian and Jain, Prateek and Sra, Suvrit and Dhillon, Inderjit S.},
title = {Information-theoretic metric learning},
year = {2007},
hideisbn = {9781595937933},
publisher = {Association for Computing Machinery},
address = {New York, NY, USA},
hideurl = {https://doi.org/10.1145/1273496.1273523},
doi = {10.1145/1273496.1273523},
booktitle = {Proceedings of the 24th International Conference on Machine Learning},
pages = {209–216},
numpages = {8},
location = {Corvalis, Oregon, USA},
series = {ICML '07}
}

@article{weinberger2005distance,
  title={Distance Metric Learning for Large Margin Nearest Neighbor Classification},
  author={Weinberger, Kilian Q and Blitzer, John and Saul, Lawrence},
  journal={Advances in Neural Information Processing Systems},
  publisher = {MIT Press},
  volume={18},
  year={2005},
url = {https://proceedings.neurips.cc/paper_files/paper/2005/file/a7f592cef8b130a6967a90617db5681b-Paper.pdf},
}

@book{bellet2015metric,
  title     = {Metric Learning},
  author    = {Aurélien Bellet and Amaury Habrard and Marc Sebban},
  publisher = {Morgan \& Claypool Publishers},
  year      = {2015},
  series    = {Synthesis Lectures on Artificial Intelligence and Machine Learning},
  volume    = {9},
  doi       = {10.2200/S00626ED1V01Y201501AIM030},
  hideurl       = {https://doi.org/10.2200/S00626ED1V01Y201501AIM030}
}

@ARTICLE{liu2024hierarchical,
  author={Liu, Shenglan and Yu, Yang and Liu, Kaiyuan and Wang, Feilong and Wen, Wujun and Qiao, Hong},
  journal={IEEE Transactions on Neural Networks and Learning Systems}, 
  title={Hierarchical Neighbors Embedding}, 
  year={2024},
  volume={35},
  number={6},
  pages={7816-7829},
  doi={10.1109/TNNLS.2022.3221103}
}

@inproceedings{beier2025joint,
  title={Joint Metric Space Embedding by Unbalanced OT with {G}romov-{W}asserstein Marginal Penalization},
  author={Beier, Florian and Piening, Moritz and Beinert, Robert and Steidl, Gabriele},
  hidedoi={10.48550/arXiv.2502.07510},
  year={2025},
  booktitle={Forty-second International Conference on Machine Learning (ICML)},
  url={https://openreview.net/forum?id=0YZHfUmsJv}
}

@unpublished{lin2025coupled,
  title={Coupled Hierarchical Structure Learning using Tree-{W}asserstein Distance},
  author={Lin, Ya-Wei Eileen and Coifman, Ronald R and Mishne, Gal and Talmon, Ronen},
  doi={10.48550/arXiv.2501.03627},
  year={2025}
}

@inproceedings{dou2022optimal,
  title={An Optimal Transport Approach to Deep Metric Learning (Student Abstract)},
  author={Dou, Jason Xiaotian and Luo, Lei and Yang, Raymond Mingrui},
  booktitle={Proceedings of the AAAI Conference on Artificial Intelligence},
  volume={36},
  number={11},
  pages={12935--12936},
  year={2022},
  doi       = {10.1609/aaai.v36i11.21604},
}

@inproceedings{wang2012supervised,
  title={Supervised Earth Mover’s Distance Learning and Its Computer Vision Applications},
  author={Wang, Fan and Guibas, Leonidas J},
  booktitle={Computer Vision--ECCV 2012: 12th European Conference on Computer Vision, Florence, Italy, October 7-13, 2012, Proceedings, Part I 12},
  pages={442--455},
  year={2012},
  publisher={Springer},
  doi={10.1007/978-3-642-33718-5_32},
}

@inproceedings{xu2018multi,
  title={Multi-Level Metric Learning via Smoothed Wasserstein Distance},
  author={Jie Xu and Lei Luo and Cheng Deng and Heng Huang},
  booktitle={Proceedings of the Twenty-Seventh International Joint Conference on Artificial Intelligence (IJCAI)},
  pages={2919--2925},
  year={2018},
  doi={10.24963/ijcai.2018/405},
}

@inproceedings{kerdoncuff2020metric,
  title={Metric Learning in Optimal Transport for Domain Adaptation},
  author={Kerdoncuff, Tanguy and Emonet, R{\'e}mi and Sebban, Marc},
  booktitle={International Joint Conference on Artificial Intelligence},
  pages={2162--2168},
  year={2020},
  organization={IJCAI},
  doi       = {10.24963/ijcai.2020/299},
}

@article{heitz2021ground,
  title={Ground metric learning on graphs},
  author={Heitz, Matthieu and Bonneel, Nicolas and Coeurjolly, David and Cuturi, Marco and Peyr{\'e}, Gabriel},
  journal={Journal of Mathematical Imaging and Vision},
  volume={63},
  pages={89--107},
  year={2021},
  publisher={Springer},
  doi={10.1007/s10851-020-00996-z},
}

@inproceedings{huang2016supervised,
 author = {Huang, Gao and Guo, Chuan and Kusner, Matt J and Sun, Yu and Sha, Fei and Weinberger, Kilian Q},
 booktitle = {Advances in Neural Information Processing Systems},
 editor = {D. Lee and M. Sugiyama and U. Luxburg and I. Guyon and R. Garnett},
 pages = {},
 publisher = {Curran Associates, Inc.},
 title = {Supervised Word Mover\textquotesingle s Distance},
 url = {https://proceedings.neurips.cc/paper_files/paper/2016/file/10c66082c124f8afe3df4886f5e516e0-Paper.pdf},
 volume = {29},
 year = {2016}
}

@article{suarez2021tutorial,
  title={A tutorial on distance metric learning: Mathematical foundations, algorithms, experimental analysis, prospects and challenges},
  author={Su{\'a}rez, Juan Luis and Garc{\'{\i}}a, Salvador and Herrera, Francisco},
  journal={Neurocomputing},
  volume={425},
  pages={300--322},
  year={2021},
  publisher={Elsevier},
  doi={10.1016/j.neucom.2020.08.017},
}

@article{luerig2024bioencoder,
author = {L{\"u}rig, Moritz D. and Di Martino, Emanuela and Porto, Arthur},
title = {BioEncoder: A metric learning toolkit for comparative organismal biology},
journal = {Ecology Letters},
volume = {27},
number = {8},
hidepages = {e14495},
doi = {10.1111/ele.14495},
hideurl = {https://onlinelibrary.wiley.com/doi/abs/10.1111/ele.14495},
hideeprint = {https://onlinelibrary.wiley.com/doi/pdf/10.1111/ele.14495},
hidenote = {e14495 ELE-00472-2024.R1},
year = {2024}
}

@ARTICLE{hu2016deep,
  author={Hu, Junlin and Lu, Jiwen and Tan, Yap-Peng},
  journal={IEEE Transactions on Circuits and Systems for Video Technology}, 
  title={Deep Metric Learning for Visual Tracking}, 
  year={2016},
  volume={26},
  number={11},
  pages={2056-2068},
  doi={10.1109/TCSVT.2015.2477936}
}

@InProceedings{coskun2018human,
author = {Coskun, Huseyin and Tan, David Joseph and Conjeti, Sailesh and Navab, Nassir and Tombari, Federico},
title = {Human Motion Analysis with Deep Metric Learning},
booktitle = {Proceedings of the European Conference on Computer Vision (ECCV)},
month = {September},
year = {2018},
doi       = {10.1007/978-3-030-01264-9_41},
}

@inproceedings{milbich2020diva,
  title={DiVA: Diverse Visual Feature Aggregation for Deep Metric Learning},
  author={Milbich, Timo and Roth, Karsten and Bharadhwaj, Homanga and Sinha, Samarth and Bengio, Yoshua and Ommer, Bj{\"o}rn and Cohen, Joseph Paul},
  booktitle={Computer Vision--ECCV 2020: 16th European Conference, Glasgow, UK, August 23--28, 2020, Proceedings, Part VIII 16},
  pages={590--607},
  year={2020},
  organization={Springer},
  doi={10.1007/978-3-030-58598-3_35},
}

@inproceedings{coria2020metric,
    title = "A Metric Learning Approach to Misogyny Categorization",
    author = "Coria, Juan Manuel  and
      Ghannay, Sahar  and
      Rosset, Sophie  and
      Bredin, Herv{\'e}",
    editor = "Gella, Spandana  and
      Welbl, Johannes  and
      Rei, Marek  and
      Petroni, Fabio  and
      Lewis, Patrick  and
      Strubell, Emma  and
      Seo, Minjoon  and
      Hajishirzi, Hannaneh",
    booktitle = "Proceedings of the 5th Workshop on Representation Learning for NLP",
    month = jul,
    year = "2020",
    address = "Online",
    publisher = "Association for Computational Linguistics",
    hideurl = "https://aclanthology.org/2020.repl4nlp-1.12/",
    doi = "10.18653/v1/2020.repl4nlp-1.12",
    pages = "89--94",
}

@INPROCEEDINGS{biswas2022geometric,
  author={Biswas, Sajib and Barao, Timothy and Lazzari, John and McCoy, Jeret and Liu, Xiuwen and Kostandarithes, Alexander},
  booktitle={2022 International Joint Conference on Neural Networks (IJCNN)}, 
  title={Geometric Analysis and Metric Learning of Instruction Embeddings}, 
  year={2022},
  volume={},
  number={},
  pages={1-8},
  doi={10.1109/IJCNN55064.2022.9892426}
}

@inproceedings{iscen2018mining,
  title={Mining on Manifolds: Metric Learning without Labels},
  author={Iscen, Ahmet and Tolias, Giorgos and Avrithis, Yannis and Chum, Ond{\v{r}}ej},
  booktitle={Proceedings of the IEEE Conference on Computer Vision and Pattern Recognition},
  pages={7642--7651},
  year={2018},
  doi       = {10.1109/CVPR.2018.00797},
}

@article{banach1922,
  author  = {Banach, Stefan},
  title   = {Sur les opérations dans les ensembles abstraits et leur application aux équations intégrales},
  journal = {Fundamenta Mathematicae},
  volume  = {3},
  pages   = {133--181},
  year    = {1922},
  doi     = {10.4064/fm-3-1-133-181},
}

@book{granas2003fixed,
  title={Fixed Point Theory},
  author={Granas, Andrzej and Dugundji, James and others},
  volume={14},
  year={2003},
  publisher={Springer},
  address = {New York, NY},
doi = {10.1007/978-0-387-21593-8},
}

@book{borsuk1967theory,
  title     = {Theory of Retracts},
  author    = {Karol Borsuk},
  publisher = {Państwowe Wydawnictwo Naukowe},
  year      = {1967},
  series    = {Monografie Matematyczne},
  volume    = {44},
  address   = {Warszawa},
  isbn      = {978-0800220815}
}

@inproceedings{cuturi2013sinkhorn,
  title={Sinkhorn Distances: Lightspeed Computation of Optimal Transport},
  author={Cuturi, Marco},
  booktitle={Advances in Neural Information Processing Systems},
  volume={26},
  publisher = {Curran Associates, Inc.},
  pages={2292--2300},
  year={2013},
  hidedoi = {10.5555/2999792.2999868},
  hideurl = {https://proceedings.neurips.cc/paper_files/paper/2013/file/af21d0c97db2e27e13572cbf59eb343d-Paper.pdf},
  url={https://papers.nips.cc/paper_files/paper/2013/hash/af21d0c97db2e27e13572cbf59eb343d-Abstract.html},
}

@inproceedings{feydy2019interpolating,
  title={Interpolating between Optimal Transport and MMD using {S}inkhorn Divergences},
  author={Feydy, Jean and S{\'e}journ{\'e}, Thibault and Vialard, Fran{\c{c}}ois-Xavier and Amari, Shun-ichi and Trouv{\'e}, Alain and Peyr{\'e}, Gabriel},
  booktitle={Proceedings of the 22nd International Conference on Artificial Intelligence and Statistics (AISTATS)},
  year={2019},
publisher =    {PMLR},
volume={89},
pages={2681--2690},
url       = {https://proceedings.mlr.press/v89/feydy19a.html}
}

@article{brouwer1911fixpoint,
author = {Brouwer, L.E.J.},
journal = {Mathematische Annalen},
pages = {97-115},
title = {{\"U}ber {A}bbildung von {M}annigfaltigkeiten},
doi = {10.1007/BF01456931},
volume = {71},
year = {1911},
}

@article{hao2021integrated,
  title={Integrated analysis of multimodal single-cell data},
  author={Hao, Yuhan and Hao, Stephanie and Andersen-Nissen, Erica and Mauck, William M and Zheng, Shiwei and Butler, Andrew and Lee, Maddie J and Wilk, Aaron J and Darby, Charlotte and Zager, Michael and others},
journal = {Cell},
volume = {184},
number = {13},
pages = {3573-3587.e29},
year = {2021},
issn = {0092-8674},
doi = {10.1016/j.cell.2021.04.048},
  publisher={Elsevier}
}

@article{maaten2008tsne,
  author  = {Laurens van der Maaten and Geoffrey Hinton},
  title   = {Visualizing Data using t-SNE},
  journal = {Journal of Machine Learning Research},
  year    = {2008},
  volume  = {9},
  number  = {86},
  pages   = {2579--2605},
  url     = {http://jmlr.org/papers/v9/vandermaaten08a.html}
}

@inproceedings{zha2001spectral,
 booktitle = {Advances in Neural Information Processing Systems},
 editor = {T. Dietterich and S. Becker and Z. Ghahramani},
 pages = {},
 publisher = {MIT Press},
title     = {Spectral Relaxation for K-means Clustering},
  author    = {Zha, Hongyuan and He, Xiaofeng and Ding, Chris and Gu, Ming and Simon, Horst D.},
  booktitle = {Advances in Neural Information Processing Systems (NeurIPS)},
  volume    = {14},
  pages     = {1057--1064},
  year      = {2001},
  url = {https://proceedings.neurips.cc/paper_files/paper/2001/file/d5c186983b52c4551ee00f72316c6eaa-Paper.pdf},
}

@unpublished{beckmann2025normalized,
  author    = {Beckmann, Matthias and Beinert, Robert and Bresch, Jonas},
  title     = {Normalized Radon Cumulative Distribution Transforms for Invariance and Robustness in Optimal Transport Based Image Classification},
  journal={arXiv},
  doi = {10.48550/arXiv.2506.08761},
  year      = {2025}
}

@inproceedings{beckmann2025maxnormalized,
  author       = {Beckmann, Matthias and Beinert, Robert and Bresch, Jonas},
  title        = {Max‑Normalized Radon Cumulative Distribution Transform for Limited Data Classification},
  booktitle    = {Scale Space and Variational Methods in Computer Vision},
  series       = {Lecture Notes in Computer Science},
  volume       = {15667},
  pages        = {241--254},
  publisher    = {Springer},
  year         = {2025},
  doi          = {10.1007/978-3-031-92366-1_19},
}

@article{chizat2018unbalanced,
  title={Unbalanced optimal transport: Dynamic and {K}antorovich formulations},
  author={Chizat, Lenaic and Peyr{\'e}, Gabriel and Schmitzer, Bernhard and Vialard, Fran{\c{c}}ois-Xavier},
  journal={Journal of Functional Analysis},
  volume={274},
  number={11},
  pages={3090--3123},
  year={2018},
  publisher={Elsevier},
  doi     = {10.1016/j.jfa.2018.03.008},
}

@book{schoelkopf2001kernellearning,
author = {Schölkopf, Bernhard and Smola, Alexander J.},
title = {Learning with Kernels: Support Vector Machines, Regularization, Optimization, and Beyond},
year = {2001},
isbn = {0262194759},
publisher = {MIT Press},
address = {Cambridge, MA, USA}
}

\appendix

\section{Supplementary Material}
\begin{theorem}[Banach's Fixed Point Theorem \cite{banach1922}]
    \label{thm:banach}
    Let $T: \R^{d} \to \R^{d}$ 
    be contractive with Lipschitz constant $L < 1$ 
    with respect to the norm $\|\cdot\|$.
    Then $T$ has a unique fixed point $x^*$,
    and the sequence $(x^t)_{t \in \N}$ generated by fixed point iteration 
    $x^{t+1} \coloneqq  T (x^t)$  
   converges for an arbitrary initialization $x^0 \in \R^{d}$ 
    converges linearly to $x^*$, meaning that
    \[
        \|x^{t+1} - x^*\| \le L \| x^t - x ^* \|
        \quad \text{for all} \quad 
        t \in \N.
    \] 
    \end{theorem}

\begin{theorem}[Brouwer's Fixed Point Theorem \cite{brouwer1911fixpoint}]
\label{thm: brouwer}
Let $T: \R^{d} \to \R^{d}$ 
be a continuous function 
and $K \subset \R^d$ be non-empty, convex, and compact such that $T : K \to K$.
Then,
there exists a fixed point $x_0 \in K$ of $T$,
meaning that $T(x_0) = x_0$.
\end{theorem}

The Perron-Frobenius Theorem can be found, e.g., in~\cite[Thm.\ 8.2.8]{horn2012matrix}.

\begin{theorem}[Perron-Frobenius Theorem \cite{horn2012matrix}]    \label{thm:fp}
Let $T\in  \R_{>0}^{d \times d}$. Then $T$ has a simple largest eigenvalue $\lambda_T > 0$ and all other eigenvalues $\lambda$ of $T$ fulfill $|\lambda| < \lambda_T$.
Furthermore, there exists, up to normalization, a unique eigenvector corresponding to $\lambda_T$, which has only positive entries.
\end{theorem}

\section{Proof of Theorem \ref{thm: existence of wasserstein singular} with mappings (\ref{eq: wasserstein peyre}) }\label{sec: proof of existence}

As mentioned in the discussion after Theorem \ref{thm: existence of wasserstein singular}, the proof consists of similar steps. We consider the sets
\[
\bb M_m^r \coloneqq \{ A \in \bb {D}_{m} \mid \norm{A}_{\infty} = 1 \text{ and } A_{k,\ell} \geq r \text{ for all } k \neq \ell \} 
\]
with $r > 0$. In analogy to $\bb K$, the operator $\q T$ maps $\bb M_m$ to itself.
\begin{lemma}[{\cite{huizing2022unsupervised}}]
Let $R_n \in \bb D_n$, $R_m \in \bb D_m$ be metric matrices \eqref{eq: metric matrices}, and let setup of Section \ref{sec: wasserstein} apply, i.e, normalization \eqref{normalization}, $\uX_i \neq \uX_j$ for all $i,j \in [n],$ $i \neq j$ and $\oX^k \neq \oX^\ell$ for all $k,\ell \in [m]$. Consider operators $\tilde{\q F}, \tilde{\q G},$ and $\tilde{\q T}$ defined in \eqref{fix2}. Then, there exists $0 < r \le 1$ such that
\[
\tilde{\q F}(\bb M_m^r) \subseteq \bb M_n^r, 
\quad \tilde{\q G}(\bb M_n^r) \subseteq \bb M_m^r
\quad \text{and} \quad
\tilde{\q T}(\bb M_m^r) \subseteq \bb M_m^r. 
\]
\end{lemma}

Since $\bb M_m^r$ includes constraint $\norm{A}_{\infty} = 1$, it is nonconvex and Brouwer's fixed point is not applicable. However, it is possible to use generalized Schauder theorem.

\begin{theorem}[Generalized Schauder theorem, 7(7.9) in \cite{granas2003fixed}] %
Let $\bb M$ be an absolute retract and $T: \bb M \to  \bb M$ be a compact map. Then $T$ has a fixed point.    
\end{theorem}

In the finite-dimensional case, absolute retracts are characterized as follows.

\begin{theorem}[V(10.5) in \cite{borsuk1967theory}] 
A finite-dimensional compact set is an absolute retract if and only if it is contractible and locally contractible.
\end{theorem}

Therefore, to conclude the proof of Theorem \ref{thm: existence of wasserstein singular}, we need to show that $\bb M_m^r$ is compact, contractible and locally contractible and $\tilde{\q T}: \bb M_m^r \to  \bb M_m^r$ is a compact map. Let us elaborate on these notions.

\begin{definition}
Let $\bb M$ be a topological space. The mapping $T: \bb M \to \bb M$ is called \textit{compact} if for all $a \in \bb M$ the preimage $T^{-1}(\{a\}) = \{ b \in \bb M: T(b) = a\}$ is compact.     
\end{definition}

\begin{definition}
A topological subspace $\bb S \subset \bb M$ is a \textit{deformation retract} of $\bb M$ onto $\bb S$, if there exists a map
\begin{equation*}
h: \bb M \times [0, 1] \to \bb M,
\end{equation*}
such that $h(a, 0) = a$, $h(a, 1) \in \bb S$ and $h(b, 1) = b$ for all $a \in \bb M$ and $b \in \bb S$.
\end{definition}

\begin{definition}
A space $\bb M$ is called 
\begin{enumerate}
\item \textit{contractible}, if there exists a point $a \in \bb M$ and deformation retract onto the singleton $\{a\}$;
\item \textit{(strongly) locally contractible}, if for every point $a \in \bb M$ and every neighborhood $\bb V \subset \bb M$ of $a$ there exists a neighborhood $\bb U \subset \bb V$ that is contractible in $\bb V$.
\end{enumerate}
\end{definition}

Now, we verify that all conditions are satisfied. Compactness of $\bb M_m^r$ follows directly from its definition. We can view $\bb M_m^r$ as an induced topological space of $\bb R^{m \times m}$ equipped with $\|\cdot\|_\infty$. Since $\tilde{\q T}$ is continuous by Proposition \ref{prop: wasserstein properties} and Theorem \ref{l:lipshcitz}, the preimage $\tilde{\q T}^{-1}(\{A\})$ is a closed subset of bounded $\bb M_m^r$ and, thus, compact.   

The remaining two properties are derived as separate lemmas.

\begin{lemma}
The set $\bb M_m^r$ is contractible.
\end{lemma}
\begin{proof}
In order to show contractibility, we consider the map
\begin{equation*}
h: \bb M_m^r \times [0, 1] \to \bb M_m^r,\ (A, t) \mapsto t A_0 + (1 - t) A
\end{equation*}
taking every point $A \in \bb M_m^r$ to the apex $A_0 \in \bb M_m^r$ with 
\begin{equation*}
A_0 \coloneqq \1_m \1_m^\tT - I_m.
\end{equation*}
We show that this map is well-defined. Since $\bb D_m$ as a convex set, $h(A,t) \in \bb D_m$ is a convex combination of $A,A_0 \in \bb D_m$. Let $(k,\ell)$ be the indices, such that $A_{k,\ell} = 1$. Then, $k \neq \ell$ and
\begin{equation*}
t (A_0)_{k,\ell} + (1 - t) A_{k,\ell} = t + (1 - t) = 1,
\end{equation*}
and by convexity of $\|\cdot\|_\infty$ it holds
\[
\|t A_0 + (1-t)A\|_\infty \le t \| A_0\|_\infty + (1-t) \|A\|_\infty
 = 1,
\]
so that $\| t A_0 + (1 - t) A \|_{\infty} = 1$ for all $0 \le t \le 1$. Moreover, for all $k \neq \ell$ we have  
\begin{equation*}
t (A_0)_{k,\ell} + (1 - t) (A)_{k,\ell} 
\geq t r + (1 - t) r 
= r.
\end{equation*}
The continuity of $h$ is straightforward and by observing that $h(A,0) = A$ and $h(A,1)= A_0$ for all $A \in \bb M_m^r$ we conclude that $\bb M_m^r$ retracts onto the singleton $\{A_0\}$.
\end{proof}

\begin{lemma}
The set $\bb M_m^r$ is (strongly) locally contractible.
\end{lemma}

\begin{proof}
Let us denote by $B_{R}(A) \coloneqq \{A' \in \R^{m \times m} \mid \norm{A - A'}_{\infty} \le R\}$ the closed $\|\cdot\|_\infty$-ball of radius $R$.


Let $A \in \bb M_m^r$ be arbitrary such that $A \neq \1_m \1_m^\tT - I_m$. For every neighborhood $\bb V \subset \bb M_m^r$ of $A$ we can find $\varepsilon > 0$ such that $B_{\varepsilon}(A) \cap \bb M_m^r \subset \bb V$. 
Define $\supp(A) \coloneqq \{(k,\ell) \in [m]^2\mid A_{k,\ell} = 1\}$ and set
\[
0 < \delta < \min \{\varepsilon, \min_{(k,\ell) \notin \supp(A)} 1 - A_{k,\ell} \}.
\]
Then, we take $\bb U \coloneq B_{\delta}(A) \cap \bb M_m^r \subset \bb V$ in the definition of local contractility. 

Next, we show that $\bb U$ is contractible to $A$.
Let $A' \in \bb U$ be arbitrary. For $(k,\ell) \notin \supp(A)$ it holds
\[
A_{k,\ell}' \le A_{k,\ell} + \delta < 1 - \delta + \delta = 1
\]
by the choice of $\delta$ and we have $\emptyset \neq \supp(A') \subset \supp(A)$. 

For the excluded case $A= \1_m \1_m^\tT - I_m$, we select $0 < \delta < \varepsilon$ and always get $\emptyset \neq \supp(A') \subset \supp(A)$.


Now, define 
\begin{equation*}
h: \bb U \times [0, 1] \to \bb U,\ (A', t) \mapsto (1 - t)A' + tA.
\end{equation*}
We show that this map is well-defined. As in the previous proof, $h(A',t) \in \bb D_m$, $\| h(A',t) \|_\infty \le 1$ and $h(A',t)_{k,\ell} \ge r$ for all $k,\ell \in [m]$. 
For $(k,\ell) \in \supp(A')$ we have
\begin{equation*}
|h(A',t)_{k,\ell}| = |(1 - t) A'_{k,\ell} + t A_{k,\ell}| = (1 - t) + t = 1
\end{equation*}
and hence, $\|h(A',t)\|_\infty = 1$. Finally, by convexity of $B_{\delta}(A)$ we have $h(A',t) \in B_{\delta}(A)$ and $h$ indeed maps to $B_{\delta}(A) \cap \bb M_m^r = \bb U$.

Again, by construction for all $A' \in \bb U$ it holds $h(A', 0) = A'$, $h(A', 1) = A$ and the set $\bb U$ deformation retracts onto $\{A\}$.  
\end{proof}

\end{document}